\documentclass[12pt]{article}
\usepackage[left=2cm,right=2cm, top=2.4cm,bottom=2.4cm,bindingoffset=0cm]{geometry}

\usepackage{tikz, titlesec, tikz-cd}
\usetikzlibrary{calc,intersections,through,backgrounds}
\usepackage{amsmath, amssymb}
\usepackage{amsthm, amsbsy}

\usepackage{amsthm}
\let\oldproofname=\proofname
\renewcommand{\proofname}{\rm\bf{\oldproofname}}
\usepackage{hyperref}

\usepackage[numbers, sort&compress]{natbib}

\usepackage{caption}
\usepackage{subcaption}

\usepackage{yhmath}
\usepackage{mathdots}

\usepackage[width=0.75\textwidth]{caption}

\usetikzlibrary{positioning, bending, arrows.meta}
\usetikzlibrary{decorations.text}
\usetikzlibrary{decorations.pathmorphing}
\usetikzlibrary{decorations.markings}

  \tikzset{->-/.style={decoration={
  markings,
  mark=at position .5 with {\arrow{>}}},postaction={decorate}}}
    \tikzset{-<-/.style={decoration={
  markings,
  mark=at position .5 with {\arrow{<}}},postaction={decorate}}}

    \tikzset{-->-/.style={decoration={
  markings,
  mark=at position .8 with {\arrow{>}}},postaction={decorate}}}
  
  \newcommand*\Bell{\ensuremath{\boldsymbol\ell}}

\newtheorem{lemma}{Lemma}[section]
\newtheorem{proposition}[lemma]{Proposition}
\newtheorem{theorem}[lemma]{Theorem}
\newtheorem{corollary}[lemma]{Corollary}

\theoremstyle{definition}
\newtheorem{remark}[lemma]{Remark}
\newtheorem{definition}[lemma]{Definition}

\newtheorem{example}[lemma]{Example}

\newtheorem{problem}[lemma]{Problem}
\newtheorem{conjecture}[lemma]{Conjecture}

\numberwithin{equation}{section} 
\makeatletter
\let\c@equation\c@lemma
\makeatother

  \newtheorem*{mthm*}{Master Theorem {\rm \cite[Theorem 2.6]{fomin2023incidences}}}

\newcommand{\C}{\mathbb{C}}
\newcommand{\D}{{\mathrm D}}
\newcommand{\R}{{\mathrm R}}
\newcommand{\F}{\mathrm{F}}

\renewcommand{\P}{\mathbb{P}}

\title{
Surface topology and incidence theorems over division rings}
\author{Anton Izosimov\thanks{
School of Mathematics and Statistics,
University of Glasgow;
e-mail: {\tt anton.izosimov@glasgow.ac.uk}
}}
\date{}

\tikzset{->-/.style={decoration={
  markings,
  mark=at position .7 with {\arrow{>}}},postaction={decorate}}}
  
  \tikzset{->>-/.style={decoration={
  markings,
  mark=at position .5 with {\arrow{>}}},postaction={decorate}}}
  
  \tikzset{-<-/.style={decoration={
  markings,
  mark=at position .3 with {\arrow{<}}},postaction={decorate}}}
  
\usetikzlibrary{angles, quotes}

\begin{document}

\maketitle

\abstract{
Incidence theorems concern configurations of points, lines, and, more generally, higher-dimensional subspaces in projective space. Broadly speaking, such theorems fall into two classes: those that hold over an arbitrary division ring, such as Desargues’ theorem, and those that hold only over fields, such as Pappus’ theorem. In this paper, we explain the topological origin of this distinction. To this end, we extend to the noncommutative setting the surface–graph approach to incidence theorems developed by Richter-Gebert, Fomin, and Pylyavskyy. We then show that theorems associated with graphs embedded on the sphere, such as Desargues’ theorem, hold over any division ring, whereas theorems corresponding to graphs embedded on surfaces of positive genus, such as Pappus’ theorem, typically hold if and only if the ground ring is a field. We also extend these results to the setting of arbitrary rings, not necessarily admitting division. 

}

\tableofcontents

\section{Introduction}
  \begin{figure}[t]
\centering
  \centering
\includegraphics[scale = 0.45]{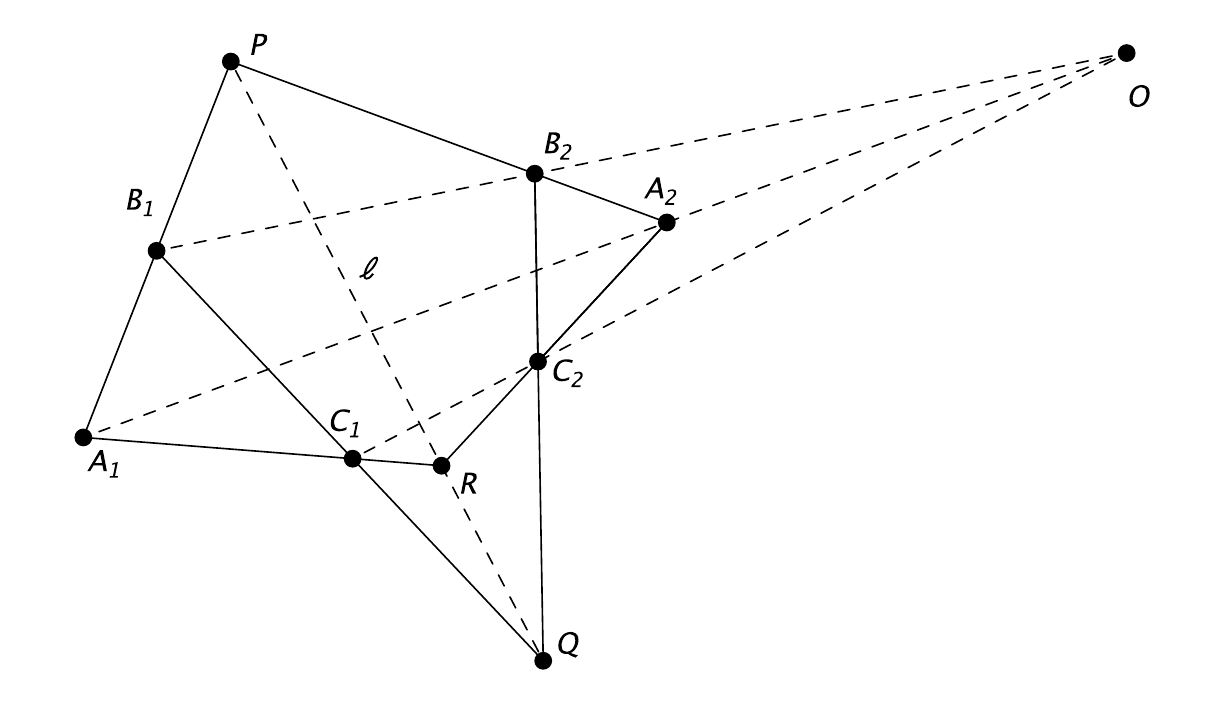}

\caption{Desargues' theorem: if two triangles $A_1B_1C_1$,  $A_2B_2C_2$ are perspective from a point (i.e., the lines $A_1A_2$, $B_1B_2$, $C_1C_2$ are concurrent), then the intersection points $A_1B_1 \cap A_2B_2$, $B_1C_1 \cap B_2C_2$, $A_1C_1 \cap A_2C_2$ of their corresponding sides are collinear. 
}\label{fig:dsrg}
\end{figure}
Incidence theorems concern configurations of points, lines, and, more generally, higher-dimensional subspaces in projective space.  The two most fundamental examples are the theorems of Desargues and Pappus, shown in Figures~\ref{fig:dsrg} and~\ref{fig:pappus}, respectively. Desargues’ theorem holds over any division ring, whereas Pappus’ theorem holds only over fields. In this paper, we explain the topological origin of this distinction.

  \begin{figure}[t!]
\centering
  \centering
\includegraphics[scale = 0.35]{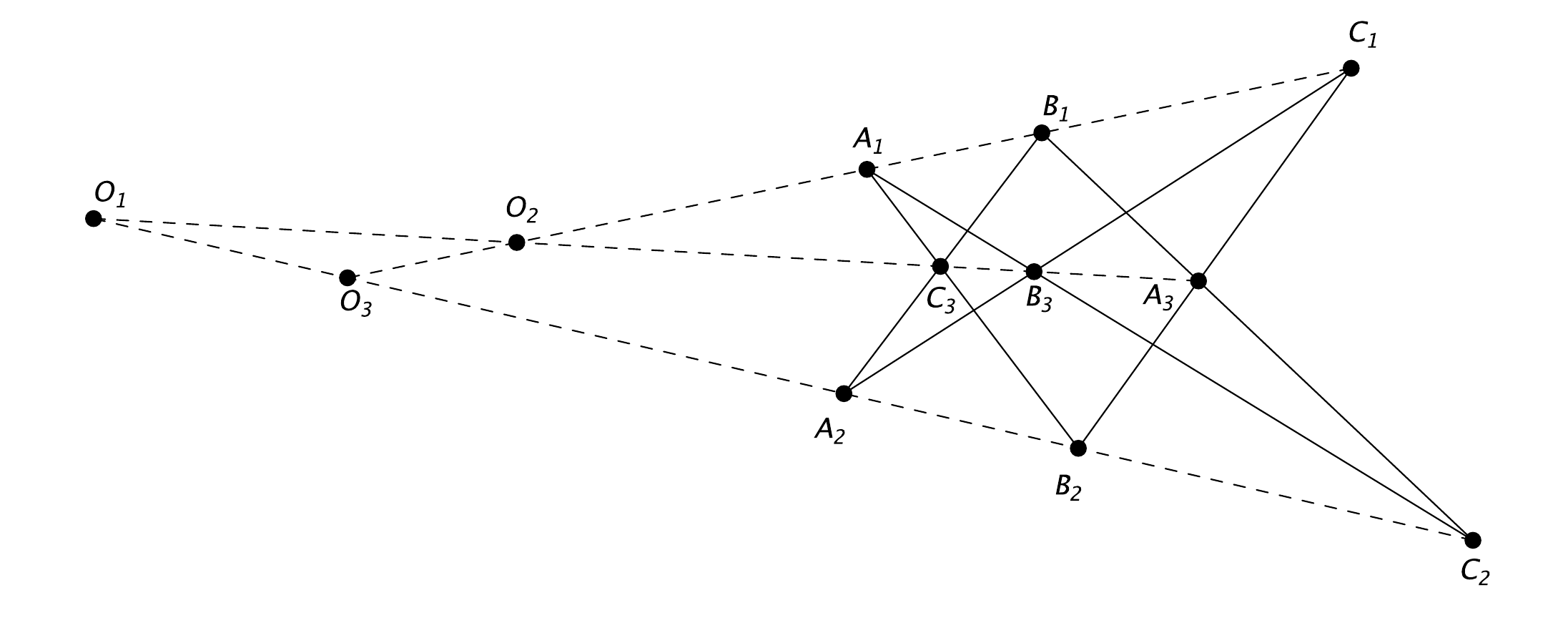}


\caption{Pappus' theorem: if both triples $A_1, B_1, C_1$, $A_2, B_2, C_2$ of non-adjacent vertices of a hexagon $A_1B_2C_1A_2B_1C_2$ are collinear, then so are the intersection points $A_1B_2 \cap A_2B_1$, $A_1C_2 \cap A_2C_1$, $B_1C_2 \cap B_2C_1$ of its opposite sides. 
}\label{fig:pappus}
\end{figure}



A topological approach to incidence theorems was developed in 
\cite{richter2006meditations, fomin2023incidences, glynn2013rabbit}. 
The main idea is to assign points or hyperplanes in a projective space to vertices 
or edges of a graph embedded in an orientable surface. For each face, one imposes 
a geometric condition involving the points or hyperplanes associated with vertices 
or edges incident to that face. While the precise condition depends on the construction, 
it can always be expressed as requiring that the product of certain invariants around 
the face equals~$1$. Moreover, when the product is computed for adjacent faces, the 
terms corresponding to their common edge cancel. Consequently, if the condition holds 
on all but one face, it must also hold on the remaining face. For each specific graph, this translates into an 
incidence theorem.

In the present paper, we generalize this construction to projective spaces over 
possibly noncommutative division rings. In this setting, the product of invariants 
around a face depends on the order of multiplication, so cancellations are not guaranteed. 
We show that the validity of the corresponding incidence theorem is governed by the topology 
of the underlying surface: theorems associated with graphs embedded on the sphere, such as 
Desargues' theorem, remain valid in the noncommutative setting, whereas theorems associated 
with graphs embedded on surfaces of higher genus, such as Pappus' theorem, \emph{typically} 
hold only over fields. The precise statement depends on the type of construction. Here, we 
consider two classes of theorems. The first class consists of theorems associated with 
\emph{polygonal subdivisions} of surfaces, generalizing Richter-Gebert's construction based 
on triangulations \cite{richter2006meditations}, and in this setting we show that all 
noncommutative theorems on surfaces of positive genus are either vacuous or fail. The second 
class consists of theorems associated with \emph{quadrilateral tilings}, as defined by Fomin 
and Pylyavskyy \cite{fomin2023incidences}, and in this setting we show that all 
noncommutative theorems on positive-genus surfaces fail in projective spaces of sufficiently 
large dimension.

We also consider two variations on the main theme. The first concerns an extension of the Fomin–Pylyavskyy construction, in which the product of invariants around a face need not be one but instead lies in a normal subgroup of the multiplicative group of the ring. The second concerns arbitrary rings, which may not be division rings. We show that our results extend to both of these settings.

\medskip

{\bf Acknowledgments.} The author is grateful to Max Planck Institute for Mathematics in
Bonn for its hospitality and financial support. This work was partially supported by the Simons Foundation through its Travel Support for Mathematicians program. 
The author also thanks Sergey Fomin, David Glynn, Pavlo Pylyavskyy, and Mikhail Skopenkov 
for valuable discussions.

\section{Incidence theorems from polygonal subdivisions}\label{sec:ps}

In this section, building upon Richter-Gebert's approach to planar incidence theorems via triangulations \cite{richter2006meditations}, we associate an incidence theorem in a projective space with any subdivision of a closed, connected, orientable surface into polygons. We then determine which of these theorems remain valid over noncommutative division rings. First, we recall the original result of \cite{richter2006meditations}.

\begin{theorem}\label{RGT}
Consider a triangulation of a closed, connected, orientable surface. Assign a point in the projective plane to each vertex and edge of the triangulation so that the following conditions hold:
\begin{enumerate}
    \item If an edge $e$ joins vertices $v$ and $w$, then the points assigned to $e$, $v$, and $w$ are collinear and distinct. 
    \item The three points assigned to the vertices of each triangle are not collinear.
    \item For all triangles except one, the points assigned to the three edges of the triangle are collinear.
\end{enumerate}
Then the points assigned to the three edges of the remaining triangle are also collinear.
\end{theorem}

This yields an infinite collection of incidence theorems, one for each triangulation.
These theorems hold over any field.

\begin{example}[Tetrahedron and Desargues, see {\cite[Section~3.2]{richter2006meditations} and \cite[Example~8.4]{fomin2023incidences}}]\label{ex:dfromt}
Consider the triangulation of the sphere with the combinatorial structure of a tetrahedron. Assign points to its vertices and edges as shown in Figure~\ref{fig:destr}, and suppose this assignment satisfies all conditions of Theorem~\ref{RGT}, with the last condition holding for all triangles except, possibly, $A_1B_1C_1$. Unpacking these conditions, we find that the lines $A_1A_2$, $B_1B_2$, and $C_1C_2$ meet at the point $O$, reproducing the configuration assumed in Desargues' theorem. Moreover, we have
$P = A_1B_1 \cap A_2B_2$, $ Q = B_1C_1 \cap B_2C_2$,  $R = A_1C_1 \cap A_2C_2,$
so that the conclusion of Theorem~\ref{RGT} -- that $P$, $Q$, and $R$ are collinear -- coincides exactly with the conclusion of Desargues' theorem. In this way, Desargues theorem is a special case of Theorem~\ref{RGT} corresponding to a tetrahedron. (More precisely, we get Desargues' theorem subject to certain general-position assumptions; for instance, the common intersection point of $A_1A_2$, $B_1B_2$, and $C_1C_2$ cannot coincide with any of the points $A_i$, $B_i$, or $C_i$, a restriction absent from Desargues' original formulation.)

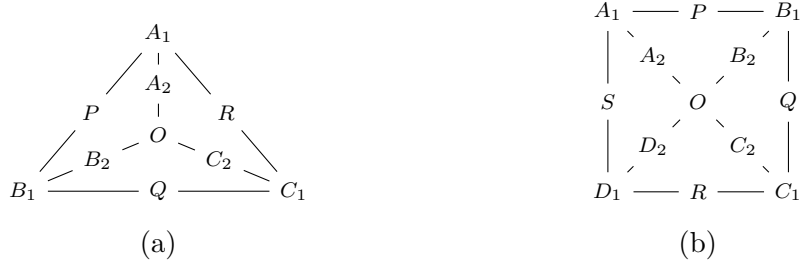
\begin{figure}[t]
\centering

\begin{subfigure}{0.4\textwidth}
\centering
\begin{tikzpicture}[scale=0.3]

\node (A1) at (0,7) {\scriptsize$ A_1$};
\node (O)  at (0,2.5) {\scriptsize$O$};
\node (B1) at (-6,0) {\scriptsize$B_1$};
\node (C1) at (6,0) {\scriptsize$C_1$};

\node (A2) at ($(O)!0.5!(A1)$) {\scriptsize$A_2$};
\node (B2) at ($(O)!0.45!(B1)$) {\scriptsize$B_2$};
\node (C2) at ($(O)!0.45!(C1)$) {\scriptsize$C_2$};

\node (ABint) at ($(A1)!0.5!(B1)$) {\scriptsize$P$};
\node (ACint) at ($(A1)!0.5!(C1)$) {\scriptsize$R$};
\node (BCint) at (0,0) {\scriptsize$Q$};

\draw (A1) -- (A2);
\draw (A1) -- (ABint);
\draw (A1) -- (ACint);
\draw (A2) -- (O);
\draw (O) -- (B2);
\draw (O) -- (C2);
\draw (ABint) -- (B1);
\draw (ACint) -- (C1);
\draw (B2) -- (B1);
\draw (C2) -- (C1);
\draw (B1) -- (BCint);
\draw (BCint) -- (C1);

\end{tikzpicture}
\caption{}
\label{fig:destr}
\end{subfigure}
\begin{subfigure}{0.4\textwidth}
\centering
\begin{tikzpicture}[scale = 0.6]
\node (A1) at (-2,2) {\scriptsize$A_1$};
\node (O)  at (0,0) {\scriptsize$O$};
\node (B1) at (2,2) {\scriptsize$B_1$};
\node (C1) at (2,-2) {\scriptsize$C_1$};
\node (D1) at (-2,-2) {\scriptsize$D_1$};

\node (A2) at ($(O)!0.5!(A1)$) {\scriptsize$A_2$};
\node (B2) at ($(O)!0.5!(B1)$) {\scriptsize$B_2$};
\node (C2) at ($(O)!0.5!(C1)$) {\scriptsize$C_2$};
\node (D2) at ($(O)!0.5!(D1)$) {\scriptsize$D_2$};

\node (ABint) at ($(A1)!0.5!(B1)$) {\scriptsize$P$};
\node (BCint) at ($(B1)!0.5!(C1)$) {\scriptsize$Q$};
\node (CDint) at ($(C1)!0.5!(D1)$) {\scriptsize$R$};
\node (DAint) at ($(D1)!0.5!(A1)$) {\scriptsize$S$};

\draw (A1) -- (A2);
\draw (A1) -- (ABint);
\draw (A1) -- (DAint);

\draw (A2) -- (O);

\draw (O) -- (B2);
\draw (O) -- (C2);

\draw (ABint) -- (B1);
\draw (DAint) -- (D1);
\draw (B2) -- (B1);
\draw (C2) -- (C1);

\draw (B1) -- (BCint);
\draw (BCint) -- (C1);

\draw (O) -- (D2) -- (D1);
\draw (C1) -- (CDint) -- (D1);

\end{tikzpicture}
\caption{
}\label{fig:destrg}
\end{subfigure}

\caption{Polygonal subdivisions of the sphere giving rise to incidence theorems: (a) the tetrahedron and Desargues' theorem; (b) the square pyramid and a generalized Desargues' theorem.}
\label{fig:desargues}

\end{figure}

\end{example}

\begin{example}[Pappus from a triangulation of the torus, see {\cite[Section 3.4]{richter2006meditations} and  \cite[Example~8.5]{fomin2023incidences}}]\label{pappus}
Consider the triangulation of the torus shown in Figure~\ref{fig:pappus-torus}, and assign points in $\mathbb{P}^2$, labeled as in the figure, to its vertices and edges. (More precisely, this is a \emph{$\Delta$-triangulation} rather than a triangulation, since in a triangulation two triangles can intersect only along a vertex or an edge.) Suppose that this assignment satisfies the conditions of Theorem~\ref{RGT}. From the first condition, we obtain that the points $A_i, B_i, C_i$ lie on the line $O_jO_k$ for any permutation $(i,j,k)\in S_3$, in agreement with the notation of Figure~\ref{fig:pappus}. The third condition, imposed on each individual triangle, implies that the points $A_i, B_j, C_k$ are collinear, again for $(i,j,k)\in S_3$. Theorem~\ref{RGT} asserts that one of these collinearities follows from the others, which is precisely the statement of Pappus’ theorem. (As in the previous example, this derivation only gives Pappus in the general-position case, since the assumptions of Theorem~\ref{RGT} are somewhat more restrictive than those of the classical theorem.)

\begin{figure}[t]
\centering

\begin{subfigure}{0.4\textwidth}
\centering
\begin{tikzpicture}[xscale=0.5, yscale=0.75]

\node at (5,5) (n11) {\scriptsize$O_2$};
\node at (8,4) (n24) {\scriptsize$O_1$};
\node at (8,2) (n44) {\scriptsize$O_2$};
\node at (5,1) (n51) {\scriptsize$O_1$};
\node at (2,2) (n41) {\scriptsize$O_2$};
\node at (2,4) (n21) {\scriptsize$O_1$};

\node at (5,3) (n32) {\scriptsize$O_3$};

\node at ($(n32)!0.5!(n11)$) (A1) {\scriptsize$A_1$};
\node at ($(n32)!0.5!(n24)$) (B2) {\scriptsize$B_2$};
\node at ($(n32)!0.5!(n44)$) (C1) {\scriptsize$C_1$};
\node at ($(n32)!0.5!(n51)$) (A2) {\scriptsize$A_2$};
\node at ($(n32)!0.5!(n41)$) (B1) {\scriptsize$B_1$};
\node at ($(n32)!0.5!(n21)$) (C2) {\scriptsize$C_2$};

\node at ($(n11)!0.5!(n21)$) (B3) {\scriptsize$B_3$};
\node at ($(n21)!0.5!(n41)$) (A3) {\scriptsize$A_3$};
\node at ($(n41)!0.5!(n51)$) (C3) {\scriptsize$C_3$};
\node at ($(n51)!0.5!(n44)$) (B3b) {\scriptsize$B_3$};
\node at ($(n44)!0.5!(n24)$) (A3b) {\scriptsize$A_3$};
\node at ($(n24)!0.5!(n11)$) (C3b) {\scriptsize$C_3$};

\draw (n32)--(A1)--(n11);
\draw (n32)--(B2)--(n24);
\draw (n32)--(C1)--(n44);
\draw (n32)--(A2)--(n51);
\draw (n32)--(B1)--(n41);
\draw (n32)--(C2)--(n21);

\draw (n11)--(B3)--(n21);
\draw (n21)--(A3)--(n41);
\draw (n41)--(C3)--(n51);
\draw (n51)--(B3b)--(n44);
\draw (n44)--(A3b)--(n24);
\draw (n24)--(C3b)--(n11);

\end{tikzpicture}
\caption{}
\label{fig:pappus-torus}
\end{subfigure}
\begin{subfigure}{0.4\textwidth}
\centering
\begin{tikzpicture}[scale=2.7]

\coordinate (V00) at (0,0);
\coordinate (V05) at (0,0.5);
\coordinate (V01) at (0,1);
\coordinate (V50) at (0.5,0);
\coordinate (V55) at (0.5,0.5);
\coordinate (V51) at (0.5,1);
\coordinate (V10) at (1,0);
\coordinate (V15) at (1,0.5);
\coordinate (V11) at (1,1);

\def\fracgap{0.35} 

\foreach \y/\xstart/\xend/\label in {
  1/0/0.5/A_2, 1/0.5/1/A_4,
  0.5/0/0.5/B_4, 0.5/0.5/1/B_2,
  0/0/0.5/A_2, 0/0.5/1/A_4
} {
    \pgfmathsetmacro{\xdiff}{\xend-\xstart}
    \pgfmathsetmacro{\xmid}{(\xstart+\xend)/2}
    \pgfmathsetmacro{\xleft}{\xmid-\xdiff*\fracgap/2}
    \pgfmathsetmacro{\xright}{\xmid+\xdiff*\fracgap/2}
    \draw (\xstart,\y) -- (\xleft,\y);
    \node at (\xmid,\y) {\scriptsize$\label$};
    \draw (\xright,\y) -- (\xend,\y);
}

\foreach \x/\ylow/\yhigh/\label in {
  0.5/0.5/1/A_3, 1/0.5/1/A_1, 1/0/0.5/B_3,
  0/0.5/1/A_1, 0.5/0/0.5/B_1, 0/0/0.5/B_3
} {
    \pgfmathsetmacro{\ydiff}{\yhigh-\ylow}
    \pgfmathsetmacro{\ymid}{(\ylow+\yhigh)/2}
    \pgfmathsetmacro{\ylowseg}{\ymid-\ydiff*\fracgap/2}
    \pgfmathsetmacro{\yhighseg}{\ymid+\ydiff*\fracgap/2}
    \draw (\x,\ylow) -- (\x,\ylowseg);
    \node at (\x,\ymid) {\scriptsize$\label$};
    \draw (\x,\yhighseg) -- (\x,\yhigh);
}

\foreach \p in {V00,V05,V01,V50,V55,V51,V10,V15,V11} {
  \fill (\p) circle (0.02);
}
\end{tikzpicture}
\caption{}
\label{fig:mobius-dandelin-gallucci}
\end{subfigure}

\caption{Polygonal subdivisions of the torus giving rise to incidence theorems: (a) Pappus' theorem; (b) M\"obius' theorem. In both panels, opposite sides of the displayed domain are identified.}
\label{fig:desargues-pappus}

\end{figure}
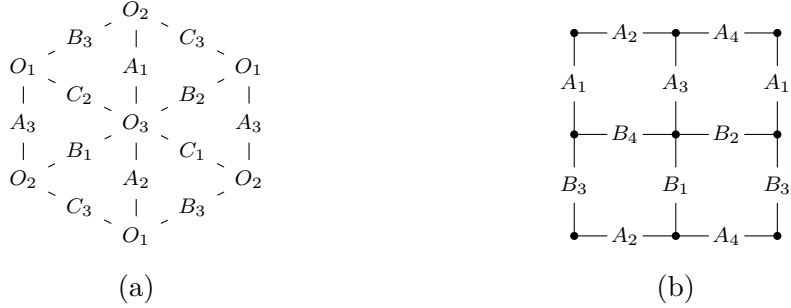


\end{example}

\begin{remark}
In examples considered throughout the paper, we will often say that for a given triangulation (or, more generally, a graph embedded on a surface), Theorem~\ref{RGT} (or one of its generalizations discussed below) specializes to a certain classical theorem. Such a specialization requires identifying the data of the classical theorem with the data assigned to the vertices and edges of the triangulation. These identifications typically rely on additional general position assumptions. Further arguments may be necessary to obtain the classical theorem in full generality. For example, see \cite[Example 2.13]{pylyavskyy2025incidences} for a derivation of Pappus's theorem from the triangulation shown in Figure~\ref{fig:pappus-torus}, which does not depend on any extra general position hypotheses.\end{remark}

Our aim is to determine which instances of Theorem~\ref{RGT} remain valid in the noncommutative setting. We address this question within the broader framework of incidence theorems associated with \emph{polygonal subdivisions}, allowing also for arbitrary dimension of the ambient projective space.

{
\begin{definition}
A \emph{polygonal subdivision} of a closed, connected, orientable surface $\Sigma$ is a cellular decomposition of $\Sigma$  in which the closure of every $2$-cell is an embedded polygon with at least two vertices.
\end{definition}
}

In what follows, $\P^n(\D)$ denotes the \emph{left} projective space of dimension $n$ over a division ring $\mathbb{D}$. All of our results hold verbatim for right projective spaces as well, which can be seen by viewing them as left projective spaces over the opposite ring.

\begin{definition}
A \emph{realization} of a polygonal subdivision   is an assignment of a point   in $\P^n(\D)$   to each vertex and edge of the subdivision such that the following conditions hold:
\begin{enumerate}
    \item  If an edge $e$ joins vertices $v$ and $w$, then the points assigned to $e$, $v$, and $w$ are collinear and distinct. 
    \item Points assigned to the vertices of each face are linearly independent.
\end{enumerate}
\end{definition}

\begin{definition}
Given a realization of a polygonal subdivision, say that the \emph{Menelaus condition} holds for a given face of the subdivision if points assigned to the edges of that face are linearly dependent.
\end{definition}

\begin{definition}\label{RGT2}
The \emph{incidence theorem  in $\P^n(\D)$  associated with a polygonal subdivision $\mathcal P$} is the following statement: for any realization of $\mathcal P$ in $\P^n(\D)$, if the Menelaus condition holds for all faces except one, then it also holds for the remaining face.
%
%
\end{definition}

Our main result in this framework is the following:

\begin{theorem}\label{mainThmPS}
Let $\mathcal P$ be a polygonal subdivision of a closed, connected, orientable surface~\(\Sigma\). Then the incidence theorem corresponding to $\mathcal P$ holds in $\P^n(\D)$ if and only if one of the following conditions is satisfied:
\begin{enumerate}
\item $\mathcal P$ has a face with more than \(n+1\) vertices;
\item \(\D\) is a field;
\item \(\Sigma\) is a sphere.
\end{enumerate}

\end{theorem}

\begin{remark}\label{algMen}
If a subdivision $\mathcal P$ has a face $F$ with more than $n+1$ vertices, then the corresponding incidence theorem in $\P^n(\D)$ is trivial, since no assignment of linearly independent points to the vertices of $F$ is possible and hence $\mathcal P$ admits no realizations. To avoid this degeneracy, one may drop the linear-independence requirement and replace the Menelaus condition with its algebraic analogue, which asserts that for an $m$-gonal face the multiratio of the points assigned to the vertices and edges of the face is equal to $(-1)^{m+1}$; see equation~\eqref{eqstar} below.  With this modification, the incidence theorem corresponding to $\mathcal P$ holds in $\P^n(\D)$ if and only if $\mathbb{D}$ is a field or the underlying surface is a sphere; see Remark~\ref{algMen2}.

\end{remark}

Theorem \ref{mainThmPS} will be established in Section~\ref{proofSec} as a consequence of general results on {connections on surface graphs} discussed in Section~\ref{secCon}. We focus on the \emph{only if} part: for any noncommutative division ring $\D$ and any polygonal subdivision $\mathcal P$ of a surface of positive genus whose faces have at most $n+1$ vertices, there exists a realization in $\P^n(\D)$ where the Menelaus condition holds on all faces except one. Moreover, we show that this exceptional face can be chosen arbitrarily. 

The \emph{if} part of Theorem~\ref{mainThmPS} can be found in the existing literature. It was first stated in~\cite{glynn2016cosmology} without proof; see Theorem~2. A proof was later given in~\cite{GlynnPreprint}. The planar case was also independently proved in~\cite{pylyavskyy2025incidences}. Here we give an independent proof based on graph connections.

%
%


\begin{example}
Desargues’ theorem (Example~\ref{ex:dfromt}) corresponds to a subdivision of the sphere and therefore holds over any division ring. Pappus' theorem (Example \ref{pappus}) corresponds to a subdivision of the torus and thus holds over a division ring $\D$ if and only if $\D$ is a field. 
\end{example}


%
%
%

{
\begin{example}[Generalized Desargues from a square pyramid]\label{ex:sp}
Consider a polygonal subdivision of the sphere with the combinatorial structure of a square pyramid. Assign points to its vertices and edges, labelling them as in Figure~\ref{fig:destrg}. Assume that this gives a realization of the subdivision, and that the Menelaus condition is satisfied for all faces except, possibly, the base of the pyramid. These conditions are equivalent to requiring that the lines \(A_1A_2\), \(B_1B_2\), \(C_1C_2\), and \(D_1D_2\) meet at the point \(O\), while the points \(P,Q,R,S\) are given by \(P = A_1B_1 \cap A_2B_2\), \(Q = B_1C_1 \cap B_2C_2\), \(R = C_1D_1 \cap C_2D_2\), and \(S = A_1D_1 \cap A_2D_2\). The conclusion of the corresponding incidence theorem is that these four points are coplanar, yielding the following generalization of Desargues’ theorem: if two (non-planar) spatial quadrilaterals \(A_1B_1C_1D_1\) and \(A_2B_2C_2D_2\) are perspective from a point -- that is, if the lines joining corresponding vertices meet at a single point -- then the intersection points of the corresponding sides lie in a common plane. A version of this result for \((n+1)\)-gons in \({\P}^n\) appears in~\cite{bell1955generalized}. The corresponding subdivision is a pyramid with an \((n+1)\)-gonal base. By Theorem \ref{mainThmPS}, this theorem holds over any division ring.
\end{example}
}
\begin{example}[Generalized Pappus from a genus $g$ surface]
In Figure~\ref{fig:pappus-torus}, replace the hexagon split into six triangles with a polygon of $4g + 2$ sides, split into $4g + 2$ triangles, and glue the opposite sides to produce a surface of genus $g$. The corresponding incidence theorem is as follows. Suppose we are given a polygon with $4g + 2$ sides such that the vertices with even indices are collinear, and the vertices with odd indices are also collinear. If $2g$ intersection points of opposite sides lie on a line, then the last intersection point of opposite sides also lies on that line. A version of this theorem, where the vertices of the polygon lie on an arbitrary conic, is due to M\"obius \cite{mobius1848verallgemeinerung}, and the version presented here can be found in \cite{fritsch2016remarks}. By Theorem~\ref{mainThmPS}, this theorem holds over a division ring $\D$ if and only if $\D$ is a field.
\end{example}

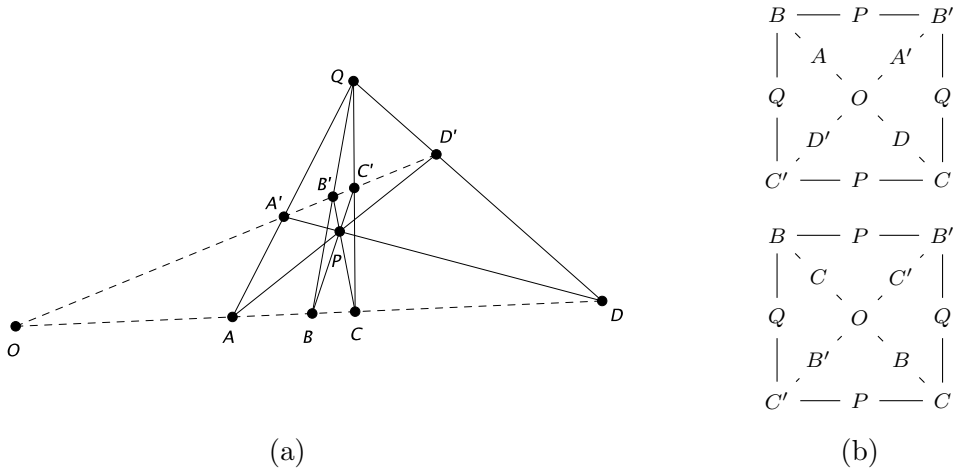
\begin{figure}[t]
\centering
\hfill
\begin{subfigure}{0.45\textwidth}
\centering
\includegraphics[scale=0.42]{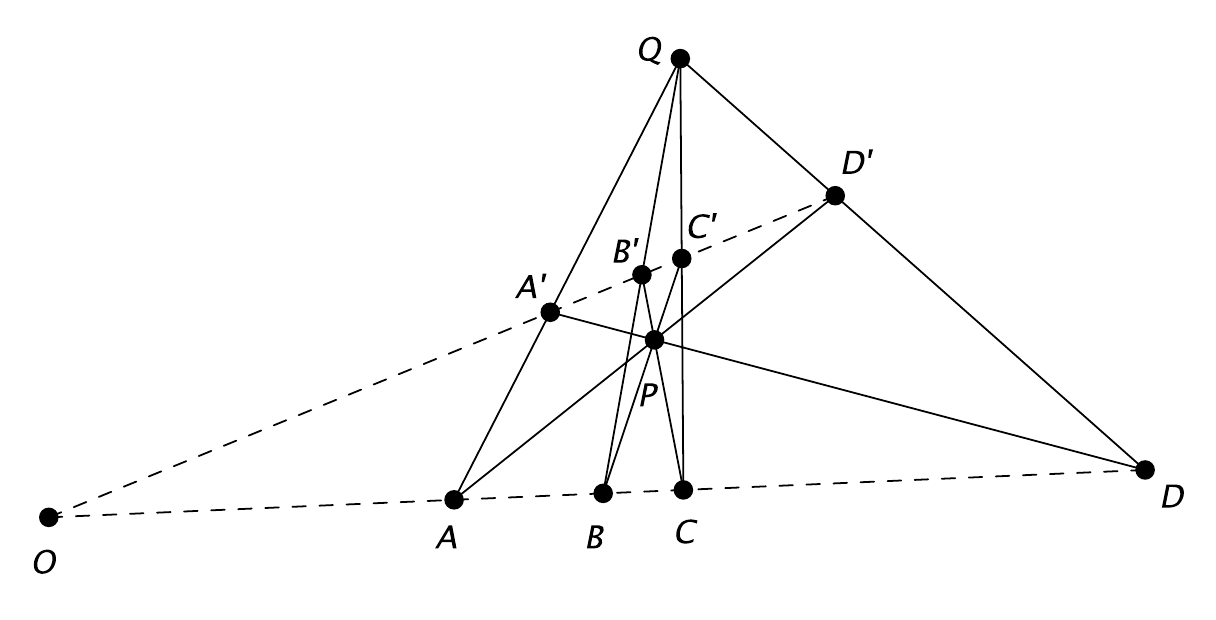}
\\ \, \\
\caption{}\label{fig:perm}
\end{subfigure}
\begin{subfigure}{0.4\textwidth}
\centering
\begin{tikzpicture}[scale=0.55]
\node (A1) at (-2,2) {\scriptsize$B$};
\node (O)  at (0,0) {\scriptsize$O$};
\node (B1) at (2,2) {\scriptsize$B'$};
\node (C1) at (2,-2) {\scriptsize$C$};
\node (D1) at (-2,-2) {\scriptsize$C'$};
\node (A2) at ($(O)!0.5!(A1)$) {\scriptsize$A$};
\node (B2) at ($(O)!0.5!(B1)$) {\scriptsize$A'$};
\node (C2) at ($(O)!0.5!(C1)$) {\scriptsize$D$};
\node (D2) at ($(O)!0.5!(D1)$) {\scriptsize$D'$};
\node (ABint) at ($(A1)!0.5!(B1)$) {\scriptsize$P$};
\node (BCint) at ($(B1)!0.5!(C1)$) {\scriptsize$Q$};
\node (CDint) at ($(C1)!0.5!(D1)$) {\scriptsize$P$};
\node (DAint) at ($(D1)!0.5!(A1)$) {\scriptsize$Q$};
\draw (A1) -- (A2);
\draw (A1) -- (ABint);
\draw (A1) -- (DAint);
\draw (A2) -- (O);
\draw (O) -- (B2);
\draw (O) -- (C2);
\draw (ABint) -- (B1);
\draw (DAint) -- (D1);
\draw (B2) -- (B1);
\draw (C2) -- (C1);
\draw (B1) -- (BCint);
\draw (BCint) -- (C1);
\draw (O) -- (D2) -- (D1);
\draw (C1) -- (CDint) -- (D1);
\end{tikzpicture}

\vspace{0.5em}

\begin{tikzpicture}[scale=0.55]
\node (A1) at (-2,2) {\scriptsize$B$};
\node (O)  at (0,0) {\scriptsize$O$};
\node (B1) at (2,2) {\scriptsize$B'$};
\node (C1) at (2,-2) {\scriptsize$C$};
\node (D1) at (-2,-2) {\scriptsize$C'$};
\node (A2) at ($(O)!0.5!(A1)$) {\scriptsize$C$};
\node (B2) at ($(O)!0.5!(B1)$) {\scriptsize$C'$};
\node (C2) at ($(O)!0.5!(C1)$) {\scriptsize$B$};
\node (D2) at ($(O)!0.5!(D1)$) {\scriptsize$B'$};
\node (ABint) at ($(A1)!0.5!(B1)$) {\scriptsize$P$};
\node (BCint) at ($(B1)!0.5!(C1)$) {\scriptsize$Q$};
\node (CDint) at ($(C1)!0.5!(D1)$) {\scriptsize$P$};
\node (DAint) at ($(D1)!0.5!(A1)$) {\scriptsize$Q$};
\draw (A1) -- (A2);
\draw (A1) -- (ABint);
\draw (A1) -- (DAint);
\draw (A2) -- (O);
\draw (O) -- (B2);
\draw (O) -- (C2);
\draw (ABint) -- (B1);
\draw (DAint) -- (D1);
\draw (B2) -- (B1);
\draw (C2) -- (C1);
\draw (B1) -- (BCint);
\draw (BCint) -- (C1);
\draw (O) -- (D2) -- (D1);
\draw (C1) -- (CDint) -- (D1);
\end{tikzpicture}
\caption{}\label{fig:permsub}
\end{subfigure}

\caption{(a) The permutation theorem: let $A,B,C,D$ and $A',B',C',D'$ be two collinear quadruples such that the lines $AA', BB', CC', DD'$ are concurrent. If three of the four lines $AD', BC', B'C, A'D$ are concurrent, then all four are concurrent.  (b) The corresponding triangulation: the square faces of the two pyramids are glued together to form an octahedron.}
\label{fig:perm-octa-vertical}
\end{figure}\begin{example}[The octahedron and the permutation theorem]\label{permoct}


The theorem depicted in Figure~\ref{fig:perm} (called the \emph{permutation theorem}  in~\cite{fomin2023incidences}) can be obtained from a triangulation of the sphere with the combinatorial structure of an octahedron. To this end, assign points to the vertices and edges of the octahedron as in Figure~\ref{fig:permsub}. This yields a realization precisely when the quadruples $\{A,B,C,D\}$ and $\{A',B',C',D'\}$ are collinear, with the corresponding lines meeting at the point $O$, while points $P$ and $Q$ are given by $P = BB' \cap CC'$, $Q = BC' \cap B'C$. Among the associated Menelaus conditions, four coincide with these realization conditions, while the remaining four assert that $P \in AA'$, $P \in BB'$, $Q \in AD'$, and $Q \in A'D$; assuming that three of these incidence relations hold, the conclusion that the fourth also holds is exactly the permutation theorem. Since the corresponding subdivision lies on a sphere, the permutation theorem holds over any division ring. (Note that the permutation theorem is not the theorem directly associated with the octahedron but rather a consequence of it, as the construction imposes additional constraints on the realization; for example, the point $P$ is assigned to two distinct edges).

\end{example}

\begin{example}[M\"obius-Dandelin-Gallucci from a quadrangulation of the torus, see {\cite[Example 8.11]{fomin2023incidences}}]\label{MDG}
The \emph{M\"obius configuration} consists of two tetrahedra inscribed in each other, that is, eight points 
$A_1, A_2, A_3, A_4, B_1, B_2, B_3, B_4 \in {\P}^3$ such that for any $(i,j,k,l) \in S_4$ we have 
$A_i \in B_jB_kB_l$ and $B_i \in A_jA_kA_l$. The \emph{M\"obius theorem}    \cite{mobius1828kann} states that any one of these 
eight incidence relations follows from the remaining seven. To encode this theorem, consider a 
quadrangulation of the torus into four quadrilaterals and assign points to its edges as shown in 
Figure~\ref{fig:mobius-dandelin-gallucci}. In order for this assignment to extend to a realization 
(i.e., to assign points to vertices so that the realization conditions hold), the quadruples 
$\{A_1, A_2, B_3, A_4\}$, $\{B_1, B_2, A_3, B_4\}$, $\{A_1, B_2, B_3, B_4\}$, and 
$\{B_1, A_2, A_3, A_4\}$ must be coplanar. Additionally, Menelaus conditions give four remaining 
coplanarity assumptions of the form $A_i \in B_jB_kB_l$ and $B_i \in A_jA_kA_l$. Thus, the 
corresponding incidence theorem is precisely the M\"obius theorem.

The M\"obius theorem can also be reformulated as what is called the \emph{sixteen points}, 
or \emph{Dandelin-Gallucci} theorem. That theorem states that if we have two quadruples of lines 
$\ell_1, \dots, \ell_4$ and $m_1, \dots, m_4$ in ${\P}^3$, such that fifteen of the sixteen 
pairs $(\ell_i, m_j)$ are coplanar, then the sixteenth pair is also coplanar. To pass from the M\"obius 
formulation to the sixteen points formulation, consider the quadruples 
$\ell_1 = A_1A_3, \ell_2 = A_2A_4, \ell_3 = B_1B_3, \ell_4 = B_2B_4$ and 
$m_1 = A_1B_3, m_2 = A_2B_4, m_3 = B_1A_3, m_4 = B_2A_4$. Then eight of the pairs $(\ell_i, m_j)$ are 
coplanar by construction, while each of the remaining eight coplanarity conditions is equivalent to a 
condition of the form $A_i \in B_jB_kB_l$ or $B_i \in A_jA_kA_l$. Thus, the sixteen points theorem 
is equivalent to the M\"obius theorem. Since the corresponding subdivision is on a torus, both theorems holds over a division ring $\D$ if and only if $\D$ is a field, cf. \cite{al1956class, horvath2019theorem}.

\end{example}

\begin{remark}
Theorem \ref{mainThmPS} and its proof given below remain valid if we replace the Menelaus condition with the \emph{Ceva condition}. For a face with vertices labeled $A_1, \dots, A_m$ and edges labeled $B_1, \dots, B_m$, where $B_i \in A_iA_{i+1}$, the Ceva condition states that the $m$ hyperplanes of the form
\[
A_1 \dots A_{i-1} B_i A_{i+2} \dots A_m
\]
are concurrent. Replacing the Menelaus condition with Ceva's condition amounts to replacing the requirement that the multiratio around the face be equal to $(-1)^{m+1}$ with the requirement that the multiratio be equal to~$1$. Note that for faces with an even number of vertices, the two conditions coincide. 

\end{remark}

{\color{blue}

}


\section{Connections on surface graphs}\label{secCon}
To prove Theorem \ref{mainThmPS}, we use graph connections. Later in the paper, we apply the same technique to prove analogues of Theorem \ref{mainThmPS} in the setting of incidence theorems defined by tilings; see Section \ref{sec:tile}.

Let $\Gamma$ be a graph. Denote by $V(\Gamma)$ the set of vertices of $\Gamma$, and by $E_0(\Gamma)$ the set of oriented edges of $\Gamma$. Each element of $E_0(\Gamma)$ is an edge of $\Gamma$ equipped with an orientation.  For any $e \in E_0(\Gamma)$, let $\bar e$ denote the same edge with the opposite orientation.  
Let also $h, t \colon E_0(\Gamma) \to V(\Gamma)$ be functions assigning to each oriented edge its {head} and {tail}, respectively.

\begin{definition} Let $G$ be a group, and $\Gamma$ be a graph. A \emph{$G$-connection} on $\Gamma$ is a function $\phi \colon E_0(\Gamma) \to G$ such that $\phi(\bar e) = \phi(e)^{-1}$.\end{definition}

\begin{definition} 
Two $G$-connections $\phi, \tilde \phi$ on a graph $\Gamma$ are \emph{gauge-equivalent} if there is a function $\psi \colon V(\Gamma) \to G$ such that for each oriented edge $e  \in E_0(G)$ we have $$\tilde \phi(e) = \psi(t(e)) \phi(e) \psi(h(e))^{-1}.$$
\end{definition}

\begin{definition}
Let $v_0 \in V(G)$ be a vertex of a graph $\Gamma$. An ordered sequence $(e_1, \dots, e_n)$ of oriented edges $e_i \in E_0(\Gamma)$ is called a \emph{loop based at} $v_0$ if 
$
h(e_i) = t(e_{i+1})$ for $i = 1, \dots, n-1,$ {and} $h(e_n) = t(e_1) = v_0.
$
\end{definition}
\begin{definition}
The \emph{holonomy} of a $G$-connection $\phi$ along a loop $\gamma = (e_1, \dots, e_n)$ based at $v_0 \in V(\Gamma)$  is \[
\mathrm{Hol}_\phi(\gamma) := \phi(e_1) \phi(e_2) \cdots \phi(e_n) \in G.
\]
\end{definition}

Clearly, the holonomy only depends on the homotopy class of a loop, yielding a well-defined group homomorphism $\mathrm{Hol}_\phi \colon \pi_1(\Gamma, v_0) \to G$. Conversely, any group homomorphism $\rho \colon \pi_1(\Gamma, v_0) \to G$ arises as the holonomy of some $G$-connection: given $\rho$, one can construct a $G$-connection on the edges of $\Gamma$ by assigning group elements arbitrarily along a spanning tree and then extending to the remaining edges so that the holonomy around each loop agrees with $\rho$. This defines a surjective map from the set of $G$-connections on $\Gamma$ to $\mathrm{Hom}(\pi_1(\Gamma), G)$. This map takes gauge-equivalent connections to conjugate homomorphisms, and conversely, conjugacy of homomorphisms implies gauge equivalence of the corresponding connections. As a result, we obtain the following:

\begin{proposition}[cf. {\cite[Section~2.4]{cimasoni2023graph} or \cite[Theorem~1]{bourque2023flat}}]\label{flatconnchar}
For any graph $\Gamma$ and group $G$, the map $\phi \mapsto \mathrm{Hol}_\phi$ induces a bijection between 
$G$-connections on $\Gamma$ modulo gauge equivalence and the quotient
$
\mathrm{Hom}(\pi_1(\Gamma), G)/G,
$
where $G$ acts on $\mathrm{Hom}(\pi_1(\Gamma), G)$ by conjugation.
\end{proposition}
This can be viewed as a graph-theoretic version of the well-known bijection between the \emph{moduli space of flat $G$-connections} and the \emph{character variety} $\mathrm{Hom}(\pi_1(\Sigma), G) / G$ for a surface~$\Sigma$. 
 

 
\begin{definition} A graph $\Gamma$ embedded in a surface $\Sigma$  is called a \emph{map} if its \emph{faces} (i.e., connected components of the complement $\Sigma \setminus \Gamma$) are homeomorphic to disks.  If $\Sigma$ has a boundary, we additionally require that every boundary component is a union of edges of $\Gamma$. \end{definition}

\begin{definition} Let $\Gamma$ be a {map} on a surface. A $G$-connection on $\Gamma$ is called \emph{flat} if its holonomies around all faces of $\Gamma$ are trivial (i.e., equal to $1 \in G$). \end{definition}
\begin{remark}
The definition of the holonomy around a face requires a choice of one of its vertices as a basepoint, as well as a choice of orientation. 
However, under a change of basepoint, the corresponding holonomy remains in the same conjugacy class, while reversing the orientation replaces the holonomy with its reciprocal. 
Therefore, the condition that the holonomy is trivial is well-defined.
\end{remark}

\begin{proposition}\label{prop:vk}
Let $\Gamma$ be a map on a surface $\Sigma$. Then, for any group $G$, 
a $G$-connection on $\Gamma$ is \emph{flat} 
if and only if the corresponding homomorphism 
$
\pi_1(\Gamma) \to G
$
factors through the natural homomorphism 
$
\pi_1(\Gamma) \to \pi_1(\Sigma).
$
\end{proposition}
\begin{proof}
This is equivalent to the statement that the kernel of the natural homomorphism 
$
\pi_1(\Gamma) \to \pi_1(\Sigma)
$
is generated by the boundaries of all faces of $\Gamma$. 
The latter is a well-known consequence of Van Kampen's theorem, see e.g. \cite[Proposition 1.26]{hatcher2005algebraic}.
\end{proof}
Combining this with Proposition \ref{flatconnchar}, we get the following:
\begin{proposition}\label{flatconnchar2}
Let $\Gamma$ be a map on a surface $\Sigma$. Then, for any group $G$, the map $\phi \mapsto \mathrm{Hol}_\phi$ induces a bijection between 
flat $G$-connections on $\Gamma$ modulo gauge equivalence and the character variety
$
\mathrm{Hom}(\pi_1(\Sigma), G)/G
$
of $\Sigma$.
\end{proposition}


\begin{corollary}\label{sphereCor}
Let $\Gamma$ be a map on the sphere $S^2$, and let $\phi$ be a $G$-connection on $\Gamma$. 
If the holonomy of $\phi$ around all but one face of $\Gamma$ is trivial, 
then the holonomy around the remaining face is also trivial 
(i.e., $\phi$ is flat).
\end{corollary}

\begin{proof}
Suppose that the holonomy is trivial around all faces except for a face $F_0$. Removing the interior of $F_0$ from $S^2$, we may view $\Gamma$ as a map on a disk $D^2$, and the connection $\phi$ can then be understood as a flat connection on this map. Flat connections on disk maps modulo gauge equivalence are in bijection with $\mathrm{Hom}(\pi_1(D^2), G)/G$, which is a one-point set. Therefore, $\phi$ is gauge-equivalent to the trivial connection which assigns the identity element $1 \in G$ to every edge of $\Gamma$. So, the holonomy of $\phi$ around any loop is trivial. In particular, $\phi$ has trivial holonomy around the boundary of the disk, i.e., the boundary of $F_0$.
\end{proof}

\begin{corollary}\label{abCor}
Let $\Gamma$ be a map on a closed, connected, orientable surface, and let $\phi$ be a $G$-connection on $\Gamma$, where $G$ is an Abelian group. 
If the holonomy of $\phi$ around all but one face of $\Gamma$ is trivial, 
then the holonomy around the remaining face is also trivial (i.e., $\phi$ is flat).
\end{corollary}

\begin{proof}
Remove the face with possibly non-trivial holonomy from the surface. This turns $\Gamma$ into a map on a surface $\Sigma_0$ with boundary. Understood as a connection on that map, $\phi$ is flat, so its holonomy defines a homomorphism $\pi_1(\Sigma_0) \to G$. Since $G$ is Abelian, this homomorphism factors through the first homology group $H_1(\Sigma_0, \mathbb{Z})$ via the natural map $\pi_1(\Sigma_0) \to H_1(\Sigma_0, \mathbb{Z})$. But the boundary of $\Sigma_0$ is trivial in $H_1(\Sigma_0, \mathbb{Z})$, so the holonomy of $\phi$ around the boundary of $\Sigma_0$ -- that is, around the removed face -- is trivial.
\end{proof}

\begin{corollary}\label{p2}
Suppose $\Gamma$ is a map on a closed, connected, orientable surface $\Sigma$ of genus $g > 0$, 
and let $G$ be a non-Abelian group. Fix a face $F_0$ of $\Gamma$. 
%
   Then there exists a $G$-connection on $\Gamma$ whose holonomy around all faces except $F_0$ is trivial, while the holonomy around $F_0$ is non-trivial.
\end{corollary}

\begin{proof}
$G$-connections on $\Gamma$ whose holonomies around all faces except $F_0$ are trivial are precisely flat $G$-connections when $\Gamma$ is viewed as a map on $\Sigma_0 = \Sigma \setminus F_0$. Thus, to show that the holonomy around $F_0$ can be non-trivial, it suffices to construct a flat connection on $\Sigma_0$ whose holonomy along its boundary is non-trivial. The surface $\Sigma_0$ is a genus $g$ surface with one boundary component, so its fundamental group is a free group on $2g$ generators. The generators $a_1, \dots, a_g, b_1, \dots, b_g$ may be chosen so that the conjugacy class in $\pi_1(\Sigma_0)$ corresponding to the boundary contains the element
\[
\gamma = a_1 b_1 a_1^{-1} b_1^{-1} \cdots a_g b_g a_g^{-1} b_g^{-1}.
\]
Choose non-commuting elements $x, y \in G$, and define a homomorphism $\rho \colon \pi_1(\Sigma_0) \to G$ by 
\[
\rho(a_1) = x, \quad \rho(b_1) = y, \quad \rho(a_2) = \rho(b_2) = \dots = \rho(a_g) = \rho(b_g) =  1.
\] 
By Proposition~\ref{flatconnchar2}, this defines a flat $G$-connection on $\Gamma$. Its holonomy along the boundary of $\Sigma_0$ -- that is, around the removed face $F_0$ -- is
\[
\rho(\gamma) = xyx^{-1}y^{-1} \neq 1,
\]
as desired.
\end{proof}

\section{Proof of the main theorem of Section \ref{sec:ps}}\label{proofSec}
In this section we prove Theorem \ref{mainThmPS}. 
\begin{proposition}\label{prop:Men}
Suppose that $A_1, \dots, A_m \in \P^n(\D)$ are linearly independent, and that for each $i$ (with indices taken modulo $m$) we are given a point
$
B_i \in A_iA_{i+1},
$
distinct from both $A_i$ and $A_{i+1}$. Let $\mathbf A_i, \mathbf B_i \in \D^{n+1}$ be lifts of $A_i$ and $B_i$, respectively, so that
\[
\mathbf B_i = \alpha_i \mathbf A_i + \beta_i \mathbf A_{i+1},
\qquad
\alpha_i,\beta_i \in \D, \quad \alpha_i, \beta_i \neq 0.
\]
Then the following hold:

\begin{enumerate}
\item The points $B_1, \dots, B_{m-1}$ span a hyperplane in the subspace $\langle A_1, \dots, A_m\rangle \subset  \P^n(\D)$.
\item This hyperplane contains none of the points $A_1, \dots, A_m$.
\item The hyperplane contains $B_m$ if and only if
\begin{equation}\label{eqstar}
(-\alpha_1^{-1}\beta_1)\cdots(-\alpha_m^{-1}\beta_m)=1 .
\end{equation}
\end{enumerate}
\end{proposition}

\begin{remark}
For $m=3$ this reduces to the non-commutative version of Menelaus' theorem \cite[Theorem~4.12]{retakh2020noncommutative}. For arbitrary $m$, it provides the non-commutative analogue of \cite[Theorem~9.12]{bobenko2008}.
\end{remark}
\begin{proof}[\bf Proof of Proposition \ref{prop:Men}]
To prove the first statement, it suffices to show that there exists a unique, up to right multiplication by elements of $\D$,   homomorphism
$
\Bell \colon \langle \mathbf A_1, \dots, \mathbf A_m \rangle \to \D
$
of left $\D$-modules such that
\[
\Bell(\mathbf B_1)=\dots=\Bell(\mathbf B_{m-1})=0.
\]
Writing $\lambda_i=\Bell(\mathbf A_i)$ and using
$
\mathbf B_i=\alpha_i \mathbf A_i+\beta_i \mathbf A_{i+1},
$
these conditions are equivalent to
\[
\alpha_i\lambda_i+\beta_i\lambda_{i+1}=0,
\qquad i=1,\dots,m-1.
\]
 The solution space of this system is a rank one right $\D$-module given by
\[
\lambda_i
=
(-\beta_{i-1}^{-1}\alpha_{i-1})
\cdots
(-\beta_1^{-1}\alpha_1)\,\lambda_1,
\]
proving the first statement. Also, since $\lambda_i\neq0$ for all $i$, none of the points $A_1,\dots,A_m$ lies in the hyperplane determined by~$\Bell$, which proves the second statement. Finally, $\Bell(\mathbf B_m)=0$ is equivalent to
\[
\alpha_m\lambda_m+\beta_m\lambda_1=0,
\]
and substituting the above expression for $\lambda_m$ gives
\[
(-\beta_m^{-1}\alpha_m)\cdots(-\beta_1^{-1}\alpha_1)=1,
\]
which is equivalent to \eqref{eqstar}.
\end{proof}

Let $\D^\times$ denote the group of units (non-zero elements) in a division ring $\D$. Given a polygonal subdivision $\mathcal P$ of a closed, connected, orientable surface $\Sigma$
and a realization of that subdivision in $\P^n(\D)$, we construct a 
$\D^{\times}$-connection on the $1$-skeleton of $\mathcal P$ as follows. First, lift all points associated with vertices and edges of $\mathcal P$ to $\D^{n+1}$. 
Let $e$ be an oriented edge of $\mathcal P$, and let $\mathbf B_e , \mathbf A_t, \mathbf A_h \in \D^{n+1}$ 
denote the vectors associated with $e$, its tail, and its head, respectively. 
By definition of a realization, there exist non-zero elements $\alpha,\beta \in \D$ such that
\[
\mathbf B_e = \alpha \mathbf A_t + \beta \mathbf A_h.
\]
We then define a $\D^{\times}$-valued connection $\phi$ on edges by
\[
\phi(e) := -\alpha^{-1} \beta.
\]
It is straightforward to check that $\phi$ is invariant under rescaling of the vectors 
assigned to edges, whereas rescaling the vectors assigned to vertices changes the connection 
by a gauge transformation. Hence, each realization of the subdivision $\mathcal P$ gives rise to a connection 
on its $1$-skeleton that is well-defined up to gauge equivalence. Furthermore, by Proposition~\ref{prop:Men}, a realization satisfies the Menelaus condition 
at a given face if and only if the holonomy of the associated connection around that face is trivial.


Now, the incidence theorem in $\P^n(\D)$ associated with 
$\mathcal P$ can be reformulated as follows: for every $\D^{\times}$-connection 
on the $1$-skeleton of $\mathcal P$ which comes from a realization in $\P^n(\D)$, 
if the holonomy around all but one face is trivial, then the holonomy around the 
remaining face is also trivial.

\begin{proposition}\label{P1}
Suppose that a polygonal subdivision $\mathcal P$ has a face with more than $n+1$ vertices.  
Then the theorem associated with $\mathcal P$ holds in $\P^n(\D)$.
\end{proposition}

\begin{proof}
In this case, $\mathcal P$ cannot be realized in $\P^n(\D)$,  
so the corresponding theorem is vacuously true.
\end{proof}

\begin{proposition}\label{P2}
Suppose that $\mathcal P$ is a polygonal subdivision of the sphere.  
Then the theorem associated with $\mathcal P$ holds in $\P^n(\D)$.
\end{proposition}

\begin{proof}
By Corollary~\ref{sphereCor}, any connection on the $1$-skeleton of $\mathcal P$ has the property that if the holonomy around all but one face is trivial, then the holonomy around the remaining face is also trivial.  
In particular, this is so for connections arising from realizations of $\mathcal P$  in $\P^n(\D)$, which implies that the corresponding theorem holds.
\end{proof}

\begin{proposition}\label{P3}
Let $\mathcal P$ be a polygonal subdivision of a closed, connected, orientable surface. Suppose that $\D$ is a field.  
Then the theorem associated with $\mathcal P$ holds in $\P^n(\D)$.
\end{proposition}

\begin{proof}
By Corollary~\ref{abCor}, any connection on the $1$-skeleton of $\mathcal P$ valued in an Abelian group has the property that if the holonomy around all but one face is trivial, then the holonomy around the remaining face is also trivial.  
In particular, this is so $\D^{\times}$-connections arising from realizations of $\mathcal P$ in $\P^n(\D)$, which implies that the corresponding theorem holds.
\end{proof}

It remains to show that if none of these conditions hold, then the theorem in $\P^n(\D)$ associated with $\mathcal P$ fails.  
We begin with the following.

\begin{lemma}\label{lem:finite_union_submodules}
A finitely generated module over an infinite division ring $\D$
cannot be represented as a finite union of its proper submodules.
\end{lemma}

\begin{proof}
We argue by induction on the rank. For rank one, the result is immediate, since the only proper submodule is the zero module. Assume the statement holds for all modules of rank at most $k$, and consider a module $M$ of rank $k+1$. Suppose that
\[
M= \bigcup_{i=1}^n M_i,
\]
where each $M_i$ is a submodule. We must show that $M_i = M$ for some $i$. Fix a basis $e_1, \dots, e_{k+1}$ in $M$. For each $a \in \D$, consider the submodule
\[
N_\alpha := \langle e_1 + ae_2, e_3, \dots, e_{k+1}\rangle \subset M.
\]
The family $\{N_\alpha\}_{\alpha \in \D}$ consists of pairwise distinct submodules. So, since $\D$ is infinite, there exists $\alpha \in \D$ such that $M_i \neq N_\alpha$ for all $i=1,\dots,n$. For this choice of $\alpha$, we have
\[
N_\alpha = \bigcup_{i=1}^n \,(M_i \cap N_\alpha).
\]
Since $N_\alpha$ has rank $k$, by the induction hypothesis some $M_i \cap N_\alpha$ equals $N_\alpha$, so $N_\alpha \subseteq M_i$. But since $M_i \neq N_\alpha$, it follows that $M_i$ has rank $k+1$, and hence $M_i = M$, as required.
\end{proof}

\begin{corollary}\label{corinf}
Suppose $\D$ is an infinite division ring and $k \ge 1$.
Then there exists an infinite subset $S \subset \D^k$ such that
any $k$-element subset of $S$ is linearly independent.
\end{corollary}
\begin{proof}
We construct $S$ inductively.  Choose $\mathbf A_1, \dots, \mathbf A_k$ to be a basis of $\D^k$. Suppose vectors $\mathbf A_1, \dots, \mathbf A_m$ have been chosen, with $m \ge k$, such that any $k$ of them are linearly independent.  
For each $(k-1)$-element subset $I \subset \{1, \dots, m\}$, let $M_I$ be the submodule of $\D^k$ spanned by $\{\mathbf A_i \mid i \in I\}$. Each $M_I$ is a proper submodule of~$\D^k$. Since there are only finitely many $(k-1)$-element subsets $I \subset \{1, \dots, m\}$, there are only finitely many such submodules $M_I$, and hence their union cannot cover all of $\D^k$.  Therefore, we may choose $\mathbf A_{m+1} \in \D^k$ outside this union. By construction, $\mathbf A_{m+1}$ does not lie in the span of any $k-1$ of the vectors $\mathbf A_1, \dots, \mathbf A_m$, so any $k$-element subset of $\{\mathbf A_1, \dots, \mathbf A_{m+1}\}$ is linearly independent.  Continuing inductively, we obtain an infinite set
$
S = \{\mathbf A_1, \mathbf A_2, \dots\}
$
with the desired property.
\end{proof}

We now complete the proof of Theorem \ref{mainThmPS}. 
\begin{proposition}\label{P4}
Let $\mathcal{P}$ be a polygonal subdivision of a closed, connected, orientable surface $\Sigma$ of positive genus, and let $\mathbb{D}$ be a non-commutative division ring. Assume that all faces of $\mathcal{P}$ have at most $n+1$ vertices. Then, for any face $F_0$ of $\mathcal{P}$, there exists a realization of $\mathcal{P}$ in $\P^n(\mathbb{D})$ such that the Menelaus condition holds on all faces except $F_0$ and fails on $F_0$.

\end{proposition}

\begin{proof}
Since $\Sigma$ has positive genus and $\mathbb{D}$ is non-commutative, Corollary~\ref{p2} guarantees the existence of a $\mathbb{D}^{\times}$-connection $\phi$ on the $1$-skeleton of $\mathcal{P}$ with trivial holonomy around all faces except $F_0$ and non-trivial holonomy around $F_0$. Let us show that there exists a realization of $\mathcal{P}$ in $\P^n(\D)$ that gives rise to this connection. Since $\mathbb{D}$ is non-commutative, it is infinite by Wedderburn’s little theorem; hence, by Corollary~\ref{corinf}, there exists an infinite subset of $\mathbb{D}^{n+1}$ such that any $n+1$ elements are linearly independent. Assign distinct vectors from this set to the vertices of $\mathcal{P}$. Then the vectors corresponding to the vertices of each face are linearly independent. For each edge $e$, fix an orientation and define
\[
\mathbf{B}_e := \mathbf{A}_t - \phi(e)\,\mathbf{A}_h,
\]
where $\mathbf{A}_t$ and $\mathbf{A}_h$ denote the vectors assigned to the tail and head of $e$, respectively. Projecting all vertex and edge vectors to $\P^n(\D)$ then yields a realization of $\mathcal{P}$ whose associated connection is $\phi$. By construction, this realization satisfies the Menelaus condition on all faces except $F_0$, where it fails.
\end{proof}

Theorem~\ref{mainThmPS} now follows directly from Propositions~\ref{P1}, \ref{P2}, \ref{P3}, and \ref{P4}.

\begin{remark}\label{algMen2}
If we drop the requirement for the points assigned to the vertices of a face to be linearly independent, as suggested in Remark \ref{algMen}, then there is no need for an infinite collection of vectors such that any $n+1$ of them are linearly independent. Instead, one may assign to the vertices of $\mathcal P$ \emph{pairwise} linearly independent vectors, and the same proof applies regardless of the number of vertices in the faces of $\mathcal P$. This shows that if the linear independence requirement is dropped and the Menelaus condition is replaced by its algebraic version~\eqref{eqstar}, then theorems corresponding to subdivisions of positive-genus surfaces fail over non-commutative rings in all dimensions.
\end{remark}

\section{Incidence theorems from quadrilateral tilings}\label{sec:tile}
We now turn to a related, though not equivalent (see Proposition \ref{CORR} below), framework for encoding incidence theorems: the language of tilings introduced by Fomin and Pylyavskyy~\cite{fomin2023incidences}.

\begin{definition}
A \emph{quadrilateral tiling} of a surface is a polygonal subdivision in which all faces are quadrilaterals and the vertices are colored black and white in such a way that every edge joins vertices of different colors. The faces of a tiling are called \emph{tiles}.
\end{definition}

In what follows, we will only consider quadrilateral tilings; they will be referred to as just \emph{tilings}.


\begin{definition}\label{def:realization}
A \emph{realization} of a tiling $\mathcal T$ in $\P^n(\D)$ is an assignment of a point in $\P^n(\D)$ to each black vertex of $\mathcal T$ and a hyperplane in $\P^n(\D)$ to each white vertex, such that for every edge of $\mathcal T$ the point and the hyperplane assigned to its endpoints are not incident.
\end{definition}

\begin{definition}\label{def:coh}
Consider two points $A_1,A_2$ and two hyperplanes $\ell_1, \ell_2$ in the projective space $\P^n(\D)$. Assume that neither point lies on either hyperplane. The quadruple $A_1, A_2, \ell_1, \ell_2$ is said to be \emph{coherent} if either $A_1 = A_2$ or $\ell_1 = \ell_2$, or else the line $A_1A_2$ and the codimension 2 subspace $\ell_1 \cap \ell_2$
have a nonempty intersection. Accordingly, given a realization of a tiling, a tile is said to be \emph{coherent} if the points $A_1,A_2$ assigned to its black vertices and the hyperplanes $\ell_1,\ell_2$ assigned to its white vertices form a coherent quadruple.
\end{definition}

The coherence condition in $\P^2$ is depicted in Figure~\ref{fig:quad}. Here we assume that $A_1 \neq A_2$ and~$\ell_1 \neq \ell_2$.

  \begin{figure}[t]
\centering
  \centering
\includegraphics[scale = 0.4]{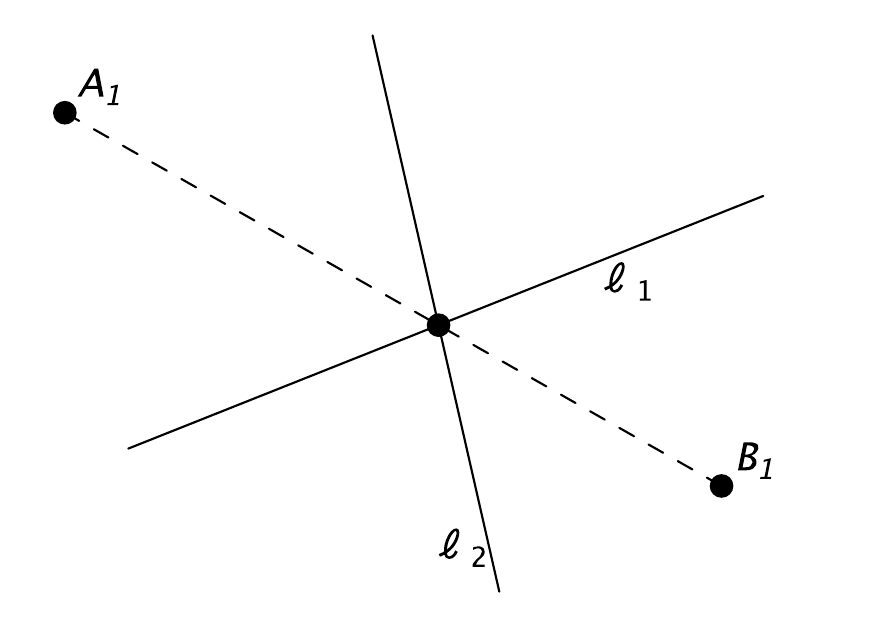}
%
%
%
%
%
%

\caption{The coherence condition in $\P^2$.
}\label{fig:quad}
\end{figure}

\begin{definition}\label{RGT2}
The \emph{incidence theorem in $\P^n(\D)$ associated with a tiling $\mathcal T$} is the following statement: for any realization of $\mathcal T$ in $\P^n(\D)$, if all tiles except one are coherent, then the remaining tile is also coherent.
\end{definition}


Theorem~2.6 of \cite{fomin2023incidences}, also called the \emph{master theorem}, asserts that for a tiling of a closed, connected, orientable surface, the associated incidence theorem holds in any dimension over any field. For extensions to arbitrary commutative rings, see also \cite{kuhne2025absolute}.






\begin{example}\label{ex:dfromtiling}
Consider a tiling of the sphere with the combinatorial structure of a cube. By assigning points and lines to its vertices as shown in Figure~\ref{fig:dsrgt}, we obtain a realization of this tiling in $\P^2$. Suppose that all tiles except the one with vertex labels $A_1, A_2C_2, C_1, \ell$ are coherent. Assume further that the points $A_1, B_1, C_1, O$ are pairwise distinct and that the lines $A_2B_2, B_2C_2, A_2C_2, \ell$ are also pairwise distinct. Under these assumptions, the assertion that the tile $A_1 b_2 C_1 \ell$ is coherent is equivalent to Desargues' theorem applied to the triangles $A_1B_1C_1$ and $A_2B_2C_2$; see \cite[Theorem~3.1]{fomin2023incidences}.
\end{example}

\begin{figure}[t]
\centering

\begin{subfigure}[t]{0.4\textwidth}
\centering
\begin{tikzpicture}[scale=1.7]
  \node (A1)     at (0,0) {\scriptsize$A_1$};
  \node (L)      at (1,0) {\scriptsize$\ell$};
  \node (B1)     at (1,1) {\scriptsize$B_1$};
  \node (A2B2)   at (0,1) {\scriptsize$A_2B_2$};
  \node (A2C2)   at (0.5,0.5) {\scriptsize$A_2C_2$};
  \node (C1)     at (1.5,0.5) {\scriptsize$C_1$};
  \node (B2C2)   at (1.5,1.5) {\scriptsize$B_2C_2$};
  \node (O)      at (0.5,1.5) {\scriptsize$O$};
  \draw (A1) -- (L) -- (B1) -- (A2B2) -- (A1);
  \draw (A2C2) -- (C1) -- (B2C2) -- (O) -- (A2C2);
  \draw (A1)   -- (A2C2);
  \draw (L)    -- (C1);
  \draw (B1)   -- (B2C2);
  \draw (A2B2) -- (O);
\end{tikzpicture}
\caption{}\label{fig:dsrgt}
\end{subfigure}
\begin{subfigure}[t]{0.4\textwidth}
\centering
\begin{tikzpicture}[xscale=0.45,yscale=0.7,
    point/.style={circle, draw, minimum size=2mm, inner sep=0pt},
    line/.style={circle, draw, fill=black, minimum size=2mm, inner sep=0pt}
]

\node[line] (n11) at (5,5) {};
\node[line] (n21) at (2,4) {};
\node[point] (n22) at (4,4) {};
\node[point] (n23) at (6,4) {};
\node[line] (n24) at (8,4) {};
\node[point] (n31) at (3,3) {};
\node[line] (n32) at (5,3) {}; 
\node[point] (n33) at (7,3) {};
\node[line] (n41) at (2,2) {};
\node[point] (n42) at (4,2) {};
\node[point] (n43) at (6,2) {};
\node[line] (n44) at (8,2) {};
\node[line] (n51) at (5,1) {};

\draw[opacity=0.4] (n11) -- (n21);
\draw[opacity=0.4] (n11) -- (n24);
\draw[opacity=0.4] (n21) -- (n41);
\draw[opacity=0.4] (n24) -- (n44);
\draw[opacity=0.4] (n41) -- (n51);
\draw[opacity=0.4] (n44) -- (n51);

\draw (n22) -- (n11);
\draw (n22) -- (n21);
\draw (n22) -- (n32);
\draw (n23) -- (n11);
\draw (n23) -- (n24);
\draw (n23) -- (n32);
\draw (n31) -- (n21);
\draw (n31) -- (n32);
\draw (n31) -- (n41);
\draw (n42) -- (n32);
\draw (n42) -- (n41);
\draw (n42) -- (n51);
\draw (n43) -- (n32);
\draw (n43) -- (n44);
\draw (n43) -- (n51);
\draw (n33) -- (n32);
\draw (n33) -- (n24);
\draw (n33) -- (n44);

\draw[opacity=0.4] (n32) -- (n11);
\draw[opacity=0.4] (n32) -- (n21);
\draw[opacity=0.4] (n32) -- (n24);
\draw[opacity=0.4] (n32) -- (n41);
\draw[opacity=0.4] (n32) -- (n51);
\draw[opacity=0.4] (n32) -- (n44);

\end{tikzpicture}
\caption{}\label{fig:pappust}
\end{subfigure}

\caption{Tilings giving rise to classical incidence theorems: (a) the cube and Desargues’ theorem; (b) a tiling of the torus corresponding to Pappus’ theorem. In (b), opposite sides of the displayed hexagonal domain are identified; grey lines indicate the edges of the corresponding triangulation.}
\label{fig:dsrgt-pappus2}
\end{figure}
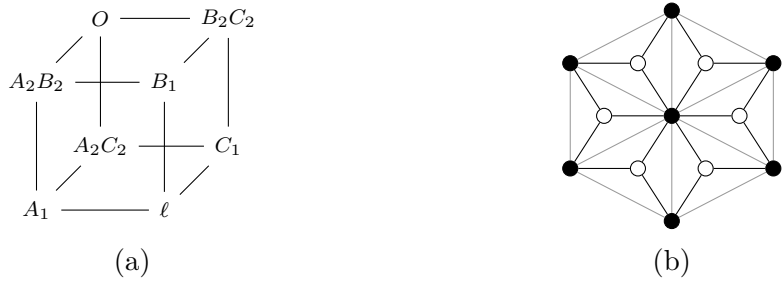
Our aim is to determine which incidence theorems associated with tilings hold over non-commutative division rings.  First, let us explain the relation between incidence theorems associated with tilings and those arising from polygonal subdivisions. 
\begin{definition}
A tiling is \emph{simple} if no two edges are incident to the same pair of vertices.
\end{definition}
\begin{proposition}
 Simple tilings, considered up to isotopy, are in one-to-one correspondence with polygonal subdivisions, also considered up to isotopy. 
 \end{proposition}
 \begin{proof}
The correspondence is as follows: black vertices of a tiling correspond to the vertices of the associated subdivision, while white vertices correspond to faces. A black vertex and a white vertex are connected by an edge whenever the corresponding vertex and face are incident. See \cite[Definition~6.1]{fomin2023incidences} and Figure~\ref{fig:dsrgt-pappus2}, which illustrates this correspondence for the triangulation giving rise to Pappus' theorem. The same construction can be applied to the triangulation corresponding to Desargues' theorem: starting with the tetrahedral triangulation of the sphere from Figure~\ref{fig:destr}, we obtain the cube shown in Figure~\ref{fig:dsrgt}.
%
%
%
%
%
%
%
%
%
%
%
%
 \end{proof}



\begin{definition}
A \emph{generic realization} of a tiling $\mathcal T$ in $\P^n(\D)$ is an assignment of a point of $\P^n(\D)$ to each black vertex of $\mathcal T$ and a hyperplane of $\P^n(\D)$ to each white vertex, such that
\begin{itemize} \item for every edge of $\mathcal T$ the point and the hyperplane assigned to its endpoints are not incident;
\item the points associated to the neighbors of every white vertex are linearly independent.
\end{itemize}
\end{definition}

\begin{definition}\label{weakThm}
The \emph{weak incidence theorem in $\P^n(\D)$ associated with a tiling $\mathcal T$} is the following statement: for any generic realization of $\mathcal T$ in $\P^n(\D)$, if all tiles except one are coherent, then the remaining tile is also coherent.
\end{definition}

\begin{proposition}\label{CORR}
Consider a simple tiling $\mathcal T$ and the associated polygonal subdivision $\mathcal P$. Then, in any projective space $\P^n(\D)$, the weak theorem associated with $\mathcal T$ is equivalent to the theorem associated with $\mathcal P$. 
\end{proposition}
\begin{remark}
One can also replace the notion of a generic realization with its dual, where the \emph{hyperplanes} associated to the neighbors of every black vertex are linearly independent. The associated dual weak theorem is then equivalent to the  theorem corresponding to the dual subdivision of $\mathcal P$.
\end{remark}
\begin{remark}
Proposition~\ref{CORR} in particular shows that any theorem associated with a tiling is stronger than the theorem associated with the corresponding subdivision. The relation between the classes of theorems associated with tilings and subdivisions is discussed in \cite[Section~4.1]{katzenberger2026manifold}.
\end{remark}
\begin{proof}[\bf Proof of Proposition \ref{CORR}]
Assume that the weak theorem in ~$\P^n(\D)$ associated with $\mathcal T$ holds. Consider a realization of $\mathcal P$ in $\P^n(\D)$ such that the Menelaus condition holds on all faces except possibly $F_0$. We show that the Menelaus condition also holds on $F_0$ by upgrading the given realization of $\mathcal P$ to a realization of $\mathcal T$ and then using the associated weak theorem.

Note that the black vertices of $\mathcal T$ are also vertices of $\mathcal P$ and thus already carry associated points in~$\P^n(\D)$. So, it remains to assign hyperplanes to the white vertices. Let $v_F$ be a white vertex corresponding to a face $F$ of $\mathcal P$.  If $F \neq F_0$, then the Menelaus condition implies that the edge-points of $F$ are linearly dependent. Moreover, by Proposition~\ref{prop:Men}, they span a hyperplane inside the span of the vertex-points, and that hyperplane contains none of the vertex-points. Therefore, there exists a hyperplane \emph{in the ambient space $\P^n(\D)$} containing all the edge-points of $F$ and none of the vertex-points. Assign this hyperplane to $v_F$.  Now, consider the case $F = F_0$. Let $A_1,\dots,A_m$ be the vertex-points and $B_1,\dots,B_m$ the edge-points of $F_0$. Again, by Proposition~\ref{prop:Men}, there exists a hyperplane containing $B_1,\dots,B_{m-1}$ and none of the $A_i$. Assign this hyperplane to $v_{F_0}$.  

With hyperplanes assigned to all white vertices, we obtain a generic realization of $\mathcal T$ in which all tiles are coherent except possibly the one corresponding to the edge of $\mathcal P$ whose corresponding point is $B_m$. By the weak theorem associated with $\mathcal T$, that last tile is also coherent. Hence, $B_m$ lies in the span of $B_1,\dots,B_{m-1}$, and the Menelaus condition holds on~$F_0$.

Conversely, assume that the theorem associated with $\mathcal P$ holds in $\P^n(\D)$. Consider a generic realization of $\mathcal T$ in $\P^n(\D)$ in which all tiles except $T_0$ are coherent. We upgrade it to a realization of $\mathcal P$ by assigning a point in $\P^n(\D)$ to each edge of $\mathcal P$, using the fact that these edges correspond to the tiles of $\mathcal T$.  

Consider a tile $T$, and let $A_1, A_2$ and $\ell_1, \ell_2$ be, respectively, the points and hyperplanes assigned to its vertices. If $T \neq T_0$, then by coherence there exists a unique point lying on the line $A_1A_2$ and on both hyperplanes $\ell_1$ and $\ell_2$. Assign this point to the corresponding edge of~$\mathcal P$.  If $T = T_0$, there is still a unique point lying on $A_1A_2$ and on $\ell_1$. Assign this point to the corresponding edge of $\mathcal P$.  

This construction yields a realization of $\mathcal P$ in which the Menelaus condition holds on all faces except possibly the one corresponding to the white vertex of $T_0$ whose associated hyperplane is $\ell_2$. Therefore, the Menelaus condition must hold on that face as well, which implies the coherence of $T_0$.
\end{proof}


%
%
%


Now, Theorem~\ref{mainThmPS} provides a characterization of tilings for which the associated weak theorem holds over division rings. In particular, this includes tilings of the sphere. Let us show that, in this setting, the corresponding strong theorem also holds.

\begin{theorem}\label{ST}
For any tiling of a sphere, the corresponding theorem in $\P^n(\D)$ holds for any division ring $\D$ and dimension $n$. 
\end{theorem}




We start with an algebraic characterization of coherence over division rings:

\begin{proposition}\label{propCoh}
Consider two points $A_1,A_2$ and two hyperplanes $\ell_1, \ell_2$ in the projective space $\P^n(\D)$. Assume that neither point lies on either hyperplane. Let $\mathbf A_1, \mathbf A_2 \in \D^{n+1}$ be vectors representing the points $A_1$ and $A_2$, and let $\Bell_1, \Bell_2 \in (\D^{n+1})^*$ be covectors representing the hyperplanes $\ell_1$ and $\ell_2$. Then the following conditions are equivalent:
\begin{enumerate}
\item The quadruple $(A_1,A_2,\ell_1,\ell_2)$ is coherent in the sense of Definition~\ref{def:coh}.
\item We have
\begin{equation}\label{eq2st}
\Bell_1(\mathbf A_1)\,\Bell_1(\mathbf A_2)^{-1}
=
\Bell_2(\mathbf A_1)\,\Bell_2(\mathbf A_2)^{-1}.
\end{equation}
\end{enumerate}
\end{proposition}

\begin{proof}
The quadruple $A_1,A_2,\ell_1,\ell_2$ is coherent if and only if the linear functions $\Bell_1, \Bell_2$ are dependent on the submodule spanned by the vectors $\mathbf A_1, \mathbf A_2$. Equivalently, the matrix
\[
\begin{pmatrix}
\Bell_1(\mathbf A_1) & \Bell_2(\mathbf A_1) \\
\Bell_1(\mathbf A_2) & \Bell_2(\mathbf A_2)
\end{pmatrix}
\]
is non-invertible. This condition is precisely equivalent to \eqref{eq2st}.
\end{proof}


\begin{proof}[\bf Proof of Theorem \ref{ST}]
As we did with polygonal subdivisions in Section~\ref{proofSec}, given a tiling $\mathcal T$ and a realization of that tiling in $\P^n(\D)$, we construct a $\D^{\times}$-connection on the $1$-skeleton of $\mathcal T$. Lift all points and hyperplanes associated with vertices of $\mathcal T$ to, respectively, vectors and covectors in $\D^{n+1}$. Let $e$ be an edge of the tiling, oriented \emph{from black to white}, and let $\mathbf A_e$ and $\Bell_e$ denote the vector and covector associated with its endpoints. We then define a $\D^{\times}$-valued connection $\phi$ on edges by
\[
\phi(e) := \Bell_e(\mathbf A_e).
\]
It is straightforward to check that rescaling the vectors and covectors assigned to vertices changes $\phi$ by a gauge transformation. Hence, each realization of a tiling $\mathcal T$ gives rise to a connection on its $1$-skeleton that is well-defined up to gauge equivalence. Furthermore, by Proposition~\ref{propCoh}, a tile is coherent if and only if the holonomy of the connection around that tile is trivial.

Now, the incidence theorem in $\P^n(\D)$ associated with $\mathcal T$ can be reformulated as follows: for every $\D^{\times}$-connection on the $1$-skeleton of $\mathcal T$ that comes from a realization in $\P^n(\D)$, if the holonomy around all but one tile is trivial, then the holonomy around the remaining tile is also trivial. For a tiling $\mathcal T$ of the sphere, this holds for any connection by Corollary~\ref{sphereCor}, so the associated theorem in $\P^n(\D)$ is always true for such $\mathcal T$.
\end{proof}

\begin{example}
As we saw in Example \ref{ex:dfromtiling}, Desargues' theorem is associated with a tiling of the sphere. Therefore, it holds over any division ring $\D$. 
\end{example}


We now consider tilings of surfaces of positive genus. 
In that setting Theorem~\ref{mainThmPS} implies the following.

\begin{theorem}\label{HGT}
Let $\mathcal T$ be a simple tiling of a surface of positive genus, and let $\D$ be a division ring. Suppose that
\[
n \ge \max_{v \in W(\mathcal T)} \bigl( \deg v \bigr) - 1,
\]
where $W(\mathcal T)$ denotes the set of white vertices of $\mathcal T$. Then the incidence theorem corresponding to $\mathcal T$ holds in $\P^n(\D)$ if and only if $\D$ is commutative.
\end{theorem}

\begin{remark}
By duality, Theorem \ref{HGT} also holds when the dimension $n$ is at least the maximum of degrees of \emph{black} vertices of $\mathcal T$ minus one.
%
\end{remark}

\begin{proof}[\bf First proof of Theorem \ref{HGT} (via polygonal subdivisions)]
If $\D$ is commutative, the incidence theorem in $\P^n(\D)$ corresponding to $\mathcal T$ holds by \cite[Theorem~2.6]{fomin2023incidences}. Thus, we may assume that $\D$ is non-commutative. Since each white vertex of $\mathcal T$ has degree at most $n+1$, the faces of the associated polygonal subdivision $\mathcal P$ contain at most $n+1$ vertices. It then follows from Theorem~\ref{mainThmPS} that the incidence theorem corresponding to $\mathcal P$ fails in $\P^n(\D)$. By Proposition~\ref{CORR}, this failure implies that the 
{weak} theorem associated with $\mathcal T$ also fails, and therefore, the {strong} theorem fails as well.\end{proof}

\begin{proof}[\bf Second proof of Theorem \ref{HGT} (direct construction)]
We give a construction that avoids polygonal subdivisions. This approach will also be useful in situations where the dimension $n$ does not satisfy the above inequality.  Assume that $\D$ is non-commutative. Fix a tile $T_0$ of $\mathcal T$ that we wish to make non-coherent. Then, by Corollary \ref{p2}, there exists a $\D^\times$-connection $\phi$ on the $1$-skeleton of $\mathcal T$ such that the holonomy around all tiles except $T_0$ is trivial, whereas
the holonomy around $T_0$ is non-trivial.

To construct a realization of $\mathcal T$ in $\P^n(\D)$ which gives rise to the connection $\phi$, take an infinite set of vectors in $\D^{\,n+1}$ such that any $n+1$ of them are linearly independent. Such a set exists by Corollary~\ref{corinf}, since non-commutative division rings are infinite. Assign distinct vectors from this set to the black vertices of $\mathcal T$. Since $\mathcal T$ is simple and each white vertex has degree at most $n+1$, the vectors assigned to the neighbors of each white vertex are linearly independent. This allows us to assign a covector to each white vertex so that, for each edge $e$ directed from a black vertex to a white vertex, we have
\[
\Bell_e(\mathbf A_e) = \phi(e),
\] 
where $\Bell_e$, $\mathbf A_e$ are the covector and vector assigned to the vertices incident to $e$. Projecting these vectors and covectors to points and hyperplanes in $\P^n(\D)$ yields a realization of $\mathcal T$ whose associated connection is $\phi$. By construction, all tiles except $T_0$ are coherent, while $T_0$ is non-coherent. Hence, the incidence theorem associated with $\mathcal T$ fails.
\end{proof}


\begin{example} 
Figure~\ref{fig:pappust} shows the tiling corresponding  to the triangulation of the torus  from Example~\ref{pappus}.  This gives a tiling formulation of Pappus' theorem, cf. \cite[first proof of Theorem~3.2]{fomin2023incidences}. Since all white vertices of this tiling have degree three and Pappus' theorem is planar, we again see that it holds over a division ring $\D$ if and only if $\D$ is a field.

\end{example}
\begin{example} Figure~\ref{fig:mobt2} shows the tiling corresponding to the quadrangulation of the torus  from Example~\ref{MDG} (the meaning of the edge labels will be explained later on).  This gives a tiling formulation of M\"obius' theorem, cf. \cite[proof of Theorem~5.5]{fomin2023incidences}. 
 Since all its white vertices have degree four, we conclude that the {M\"obius theorem} holds in a division ring $\D$ if and only if $\D$ is a field. 
\end{example}

%
%
%
%
%
%

\begin{figure}[t]
    \centering
    \begin{tikzpicture}[scale=0.9]
        \begin{scope}[xshift=0cm]
            \draw (0,0) grid (2,2);

            \foreach \x in {0,1,2}
                \foreach \y in {0,1,2}
                {
                    \pgfmathtruncatemacro{\parity}{mod(\x+\y,2)}
                    \ifnum\parity=0
                        \fill[black] (\x,\y) circle (3pt);
                    \else
                        \fill[white,draw=black] (\x,\y) circle (3pt);
                    \fi
                }

            \node[left] at (0,1.5) {\scriptsize $e$};
            \node[left] at (0,0.5) {\scriptsize $e'$};

            \node[right] at (2,1.5) {\scriptsize $e$};
            \node[right] at (2,0.5) {\scriptsize $e'$};
        \end{scope}

        \draw[->, thick] (3,1) -- (4,1);

        \begin{scope}[xshift=5cm, yshift=1cm]  
            \coordinate (H1) at (0,0);
            \coordinate (H2) at (1,0);
            \coordinate (H3) at (2,0);
            \coordinate (H4) at (3,0);
            \coordinate (H5) at (4,0);

            \coordinate (Vtop) at (2,1);
            \coordinate (Vbottom) at (2,-1);

            \draw (H1) -- (H2) -- (H3) -- (H4) -- (H5);

            \draw (H3) -- (Vtop);
            \draw (H3) -- (Vbottom);

            \draw (H1) -- (Vtop);
            \draw (H1) -- (Vbottom);
            \draw (H5) -- (Vtop);
            \draw (H5) -- (Vbottom);

            \fill[white,draw=black] (H1) circle (3pt); 
            \fill[black] (H2) circle (3pt);            
            \fill[white,draw=black] (H3) circle (3pt); 
            \fill[black] (H4) circle (3pt);            
            \fill[white,draw=black] (H5) circle (3pt); 

            \fill[black] (Vtop) circle (3pt);
            \fill[black] (Vbottom) circle (3pt);
        \end{scope}
    \end{tikzpicture}
    \caption{A non-simple tiling and its simplification. For the tiling on the left, the opposite sides of the square are identified to form a torus. For the tiling on the right, the top and bottom boundaries are identified to form a sphere.}
    \label{fig:tiling-comparison-edges-e}
\end{figure}
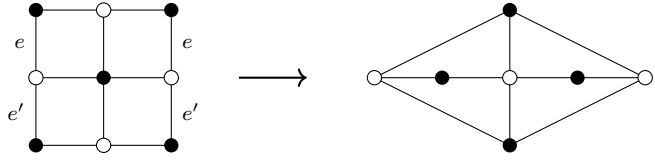
\begin{remark} 
Note that for a non-simple tiling of a surface of positive genus, the corresponding theorem may hold over non-commutative division rings in all dimensions. For example, this occurs for the torus tiling shown in the left panel of Figure~\ref{fig:tiling-comparison-edges-e}. Indeed, for this tiling, coherence of any tile is equivalent to coherence of any other tile (cf. \cite[Example 2.10]{fomin2023incidences}).

Furthermore, any non-simple tiling can be transformed into a simple one via repeated application of the following procedure. Suppose $e$ and $e'$ are edges connecting the same pair of vertices. Cut the surface along the loop formed by $e$ and $e'$, producing a new surface with two boundary components, each consisting of a copy of $e$ and a copy of $e'$. Gluing together the copies of $e$ and $e'$ lying on the same boundary component yields a new tiling of a closed surface. The incidence theorems corresponding to this new tiling imply those associated with the original tiling.

For instance, applying this procedure to the edges $e$ and $e'$ of the torus tiling in the left panel of Figure~\ref{fig:tiling-comparison-edges-e} produces the tiling of the sphere shown in the right panel. This again confirms that the theorem corresponding to the original tiling holds over all division rings and in all dimensions, despite the tiling being defined on a torus.\end{remark}

\section{Tiling theorems in low dimensions}

In the setting of polygonal subdivisions, Theorem \ref{mainThmPS} provides a complete classification of theorems that are valid over a division ring. For theorems corresponding to tilings, however, the situation is more intricate. Specifically, for a tiling $\mathcal T$ of a surface of positive genus, the corresponding theorem may or may not hold over non-commutative rings in dimensions
\[
n \; < \; \displaystyle \max\limits_{v \in W(\mathcal T)} \bigl( \deg v \bigr) - 1,
\]
where $W(\mathcal T)$ denotes the set of white vertices of $\mathcal T$.



\begin{figure}[b!]
    \centering
    \begin{tikzpicture}[scale = 1.4,
        x=1pt,y=1pt,
        whitevertex/.style={circle, draw, fill=white,minimum size=2mm, inner sep=0pt},
        blackvertex/.style={circle, fill, minimum size=2mm, inner sep=0pt}
    ]


    \node[blackvertex] (v1) at (60,0) {};
    \node[whitevertex] (v2) at (30,15) {};
    \node[whitevertex] (v3) at (90,15) {};
    \node[blackvertex] (v4) at (0,30) {};
    \node[blackvertex] (v5) at (120,30) {};
    \node[whitevertex] (v6) at (0,50) {};
    \node[whitevertex] (v7) at (40,50) {};
    \node[whitevertex] (v8) at (80,50) {};
    \node[whitevertex] (v9) at (120,50) {};
    \node[blackvertex] (v10) at (20,50) {};
    \node[blackvertex] (v11) at (100,50) {};
    \node[blackvertex] (v12) at (0,70) {};
    \node[blackvertex] (v13) at (120,70) {};
    \node[whitevertex] (v14) at (30,85) {};
    \node[whitevertex] (v15) at (90,85) {};
    \node[blackvertex] (v16) at (60,100) {};
    
      \fill[gray!20, opacity=0.5](v7.center) -- (v1.center) -- (v8.center) -- (v16.center) --  (v7.center) ;

    \node[blackvertex] (v1) at (60,0) {};
    \node[whitevertex] (v2) at (30,15) {};
    \node[whitevertex] (v3) at (90,15) {};
    \node[blackvertex] (v4) at (0,30) {};
    \node[blackvertex] (v5) at (120,30) {};
    \node[whitevertex] (v6) at (0,50) {};
    \node[whitevertex] (v7) at (40,50) {};
    \node[whitevertex] (v8) at (80,50) {};
    \node[whitevertex] (v9) at (120,50) {};
    \node[blackvertex] (v10) at (20,50) {};
    \node[blackvertex] (v11) at (100,50) {};
    \node[blackvertex] (v12) at (0,70) {};
    \node[blackvertex] (v13) at (120,70) {};
    \node[whitevertex] (v14) at (30,85) {};
    \node[whitevertex] (v15) at (90,85) {};
    \node[blackvertex] (v16) at (60,100) {};


    \draw (v1) -- (v2) node[midway,below] {\scriptsize$b$};
    \draw (v1) -- (v3) node[midway,below] {\scriptsize$a^{-1}$};

    \draw (v4) -- (v6) node[midway,left] {\scriptsize$1$};
    \draw (v4) -- (v2) node[midway,below] {\scriptsize$1$};
    \draw (v5) -- (v9) node[midway,right] {\scriptsize$1$};
    \draw (v5) -- (v3) node[midway,below] {\scriptsize$1$};

    \draw (v6) -- (v10) node[midway,below] {\scriptsize$1$};
    \draw (v10) -- (v7) node[midway,below] {\scriptsize$1$};
    \draw (v8) -- (v11) node[midway,below] {\scriptsize$1$};
    \draw (v11) -- (v9) node[midway,below] {\scriptsize$1$};

    \draw (v6) -- (v12) node[midway,left] {\scriptsize$1$};
    \draw (v9) -- (v13) node[midway,right] {\scriptsize$1$};

    \draw (v12) -- (v14) node[midway,above] {\scriptsize$a^{-1}$};
    \draw (v13) -- (v15) node[midway,above] {\scriptsize$b$};

    \draw (v14) -- (v16) node[midway,above] {\scriptsize$1$};
    \draw (v15) -- (v16) node[midway,above] {\scriptsize$1$};

    \draw (v7) -- (v16) node[midway,left] {\scriptsize$a$};
    \draw (v8) -- (v16) node[midway,right] {\scriptsize$b^{-1}$};
    \draw (v7) -- (v1) node[midway,left] {\scriptsize$b$};
    \draw (v8) -- (v1) node[midway,right] {\scriptsize$a^{-1}$};
    \draw (v2) -- (v10) node[midway,right] {\scriptsize$1$};
    \draw (v3) -- (v11) node[midway,left] {\scriptsize$1$};
    \draw (v14) -- (v10) node[midway,right] {\scriptsize$a^{-1}$};
    \draw (v15) -- (v11) node[midway,left] {\scriptsize$b$};

    \node[left=4pt]  at (v12) {\scriptsize$A_1$};
    \node[left=4pt]  at (v4)  {\scriptsize$A_2$};
    \node at ([xshift=-5pt,yshift=5pt]v10) {\scriptsize$A_3$};
    \node at ([xshift=5pt,yshift=5pt]v11) {\scriptsize$A_4$};
    \node[above=4pt] at (v16) {\scriptsize$A_2$};
    \node[below=4pt] at (v1)  {\scriptsize$A_1$};
    \node[right=4pt] at (v13) {\scriptsize$A_1$};
    \node[right=4pt] at (v5)  {\scriptsize$A_2$};

    \node[above left=2pt]  at (v14) {\scriptsize$\ell_1$};
    \node[left=2pt]        at (v6)  {\scriptsize$\ell_3$};
    \node[above right=2pt] at (v7)  {\scriptsize$\ell_4$};
    \node[above right=2pt] at (v8)  {\scriptsize$\ell_5$};
    \node[right=2pt]       at (v9)  {\scriptsize$\ell_3$};
    \node[below=2pt]       at (v2)  {\scriptsize$\ell_2$};
    \node[below=2pt]       at (v3)  {\scriptsize$\ell_1$};
    \node[above right=2pt] at (v15) {\scriptsize$\ell_2$};

    \end{tikzpicture}
    \caption{Another torus tiling corresponding to Pappus' theorem. Opposite
sides of the shown hexagonal fundamental domain should be glued to each other.}
    \label{fig:tiling_example}
\end{figure}
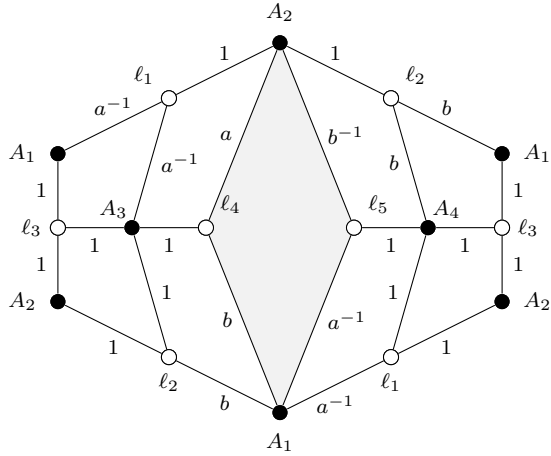

\begin{example}\label{Pappus2}
Theorem~\ref{HGT} guarantees that theorems corresponding to the tiling in Figure~\ref{fig:tiling_example} are equivalent to the commutativity of the ground ring in dimensions three or higher. However, in this case, the equivalence also holds in dimension two, because the corresponding planar theorem is again Pappus' theorem (see \cite[second proof of Theorem~3.2]{fomin2023incidences}). Let us explain why this is the case from the perspective of graph connections. Given a non-commutative division ring $\D$, we aim to construct a realization of the given tiling  in $\P^2(\D)$ such that all tiles except the shaded one are coherent, while the shaded tile is not. To this end, choose a connection on the $1$-skeleton of the tiling that has trivial holonomy around all tiles except the shaded one. Any such connection is gauge-equivalent to the one shown in Figure~\ref{fig:tiling_example} (where the edge labels represent the values of the connection on the edge oriented from black to white). This is because a flat connection on a torus with a disk removed is determined by its holonomies around two independent cycles. The connection depicted in the figure has holonomies $a,b \in \D^\times$ around these independent cycles, so one can find a connection of this form in every gauge-equivalence class.

Now, as in the proof of Theorem \ref{HGT}, we want to construct a realization in $\P^2(\D)$ that gives rise to this connection. That is, we need to assign vectors $\mathbf{A}_1, \dots, \mathbf{A}_4 \in \D^3$ to the black vertices and covectors $\Bell_1, \dots, \Bell_5 \in (\D^3)^*$ to the white vertices, as shown in the figure, so that the pairings $\Bell_i(\mathbf{A}_j)$ coincide with the edge labels. First, choose linearly independent vectors $\mathbf{A}_1, \mathbf{A}_2, \mathbf{A}_3$, and then, using the connection, determine the covectors $\Bell_1, \dots, \Bell_4$; for instance, $\Bell_1$ is determined by
\[
\Bell_1(\mathbf{A}_1) = a^{-1}, \quad 
\Bell_1(\mathbf{A}_2) = 1, \quad 
\Bell_1(\mathbf{A}_3) = a^{-1}.
\] 
Next, consider the matrix of pairings with entries $\Bell_j(\mathbf{A}_i)$ for $i,j = 1,2,3$:
\[
\begin{pmatrix}
a^{-1} & b & 1 \\
1 & 1 & 1 \\
a^{-1} & 1 & 1
\end{pmatrix}.
\] 
Using non-commutative Gaussian elimination (where rows may be multiplied by scalars on the left and columns by scalars on the right), we see that this matrix is invertible whenever $a,b \neq 1$.
 So, if this assumption holds, the covectors $\Bell_1, \Bell_2, \Bell_3$ are linearly independent, which in turn allows one to determine $\mathbf{A}_4$ using the connection. One then checks that the matrix of pairings between $\Bell_1, \Bell_2, \Bell_3$ and $\mathbf{A}_1, \mathbf{A}_2, \mathbf{A}_4$ is also invertible, again assuming $a,b \neq 1$.  So, the vectors $\mathbf{A}_1, \mathbf{A}_2, \mathbf{A}_4$ are linearly independent, allowing one to determine $\Bell_5$. This construction produces a realization for any $a,b \neq 1$. Since the holonomy around the shaded face is $aba^{-1}b^{-1}$, for any non-commuting $a,b$ the resulting realization has the property that all tiles except the shaded one are coherent, while the shaded tile is not. So, the associated theorem in dimension two fails over any non-commutative division ring.


\end{example}

We note that the key part of the argument in Example~\ref{Pappus2} relies on the fact that the graph in Figure~\ref{fig:tiling_example} can be constructed inductively by connecting each new vertex to at most three existing vertices. Such graphs are known as \emph{$3$-degenerate}. More generally, a graph is \emph{$k$-degenerate} if every subgraph has a vertex of degree at most~$k$. Equivalently, a $k$-degenerate graph is a graph that may be constructed by successively connecting new vertices to at most $k$ existing vertices. For such graphs, one may hope to use a version of the above argument to prove the following stronger version of Theorem \ref{HGT}.

\begin{conjecture}
Let $\mathcal T$ be a simple, $(n+1)$-degenerate tiling of a surface of positive genus, and let $\D$ be a division ring. Then the incidence theorem corresponding to $\mathcal T$ holds in $\P^n(\D)$ if and only if $\D$ is commutative. 

\end{conjecture}

We note that for every genus $g$ there exists a constant $k_g$ such that any graph embeddable in a genus $g$ surface is $k_g$-degenerate. In particular, any bipartite graph without multiple edges which embeds on a torus  is $4$-degenerate. As a consequence, the above conjecture would imply the following.

\begin{conjecture}
For any $g > 0$, there exists a constant $n_g$ such that, for any simple tiling $\mathcal T$ of a genus $g$ surface, the corresponding  incidence theorem over a division ring $\D$ is equivalent to commutativity of $\D$
in any dimension $n \geq n_g$. In particular, one has $n_1 = 3$, so that any theorem in dimension $n \geq 3$ corresponding to a tiling of the torus holds over a division ring $\D$ if and only if $\D$ is commutative.

\end{conjecture}


We note that these conjectures still allow for planar theorems corresponding to torus tilings that are valid over non-commutative division rings. Assuming the conjectures hold, such theorems must necessarily arise from tilings that are not $3$-degenerate. 
Below, we provide examples coming from toric embeddings of the complete bipartite graph $K_{4,4}$. For such tilings, the corresponding planar theorems turn out to be valid over any division ring.

\begin{figure}[t]
\centering

\begin{subfigure}{0.4\textwidth}
\centering
\begin{tikzpicture}[scale = 1,   
    whitevertex/.style={circle, draw, fill=black, minimum size=2mm, inner sep=0pt},
    blackvertex/.style={circle, draw, fill=white, minimum size=2mm, inner sep=0pt},
    dot/.style={dotted}, 
    line/.style={}
]

\draw[dashed] (0,0) -- (0,4) -- (4,4) -- (4,0) -- cycle;

\node[whitevertex] (W1) at (1,0) {};
\node[whitevertex] (W2) at (1,2) {};
\node[whitevertex] (W3) at (1,4) {};
\node[whitevertex] (W4) at (3,0) {};
\node[whitevertex] (W5) at (3,2) {};
\node[whitevertex] (W6) at (3,4) {};

\node[below] at (W1.south) {\scriptsize$A_1$};
\node[above] at (W2.north) {\scriptsize$A_3$};
\node[above] at (W3.north) {\scriptsize$A_1$};
\node[below] at (W4.south) {\scriptsize$A_2$};
\node[above] at (W5.north) {\scriptsize$A_4$};
\node[above] at (W6.north) {\scriptsize$A_2$};

\node[blackvertex] (B1) at (0,1) {};
\node[blackvertex] (B2) at (0,3) {};
\node[blackvertex] (B3) at (2,1) {};
\node[blackvertex] (B4) at (2,3) {};
\node[blackvertex] (B5) at (4,1) {};
\node[blackvertex] (B6) at (4,3) {};

\node[left] at (B1.west) {\scriptsize$\ell_3$};
\node[left] at (B2.west) {\scriptsize$\ell_1$};
\node[below] at (B3.south) {\scriptsize$\ell_4$};
\node[above] at (B4.north) {\scriptsize$\ell_2$};
\node[right] at (B5.east) {\scriptsize$\ell_3$};
\node[right] at (B6.east) {\scriptsize$\ell_1$};

\fill[gray!20, opacity=0.5] (W2.center) -- (B3.center) -- (W5.center) -- (B4.center) -- cycle;

\draw[line] (W1) -- (B1) node[midway,above] {\scriptsize$1$};
\draw[line] (B1) -- (W2) node[midway,below] {\scriptsize$1$};
\draw[line] (W2) -- (B2) node[midway,above] {\scriptsize$a$};
\draw[line] (B2) -- (W3) node[midway,below] {\scriptsize$1$};
\draw[line] (W3) -- (B4) node[midway,below] {\scriptsize$1$};
\draw[line] (B4) -- (W2) node[midway,above] {\scriptsize$a$};
\draw[line] (W2) -- (B3) node[midway,below] {\scriptsize$b$};
\draw[line] (B3) -- (W1) node[midway,above] {\scriptsize$b$};

\draw[line] (W4) -- (B3) node[midway,above] {\scriptsize$1$};
\draw[line] (B3) -- (W5) node[midway,below] {\scriptsize$a^{-1}$};
\draw[line] (W5) -- (B4) node[midway,above] {\scriptsize$\quad b^{-1}$};
\draw[line] (B4) -- (W6) node[midway,below] {\scriptsize$\,\, b^{-1}$};
\draw[line] (W6) -- (B6) node[midway,below] {\scriptsize$1$};
\draw[line] (B6) -- (W5) node[midway,above] {\scriptsize$1$};
\draw[line] (W5) -- (B5) node[midway,below] {\scriptsize$a^{-1}\quad$};
\draw[line] (B5) -- (W4) node[midway,above] {\scriptsize$1$};

\node[whitevertex] (W1) at (1,0) {};
\node[whitevertex] (W2) at (1,2) {};
\node[whitevertex] (W3) at (1,4) {};
\node[whitevertex] (W4) at (3,0) {};
\node[whitevertex] (W5) at (3,2) {};
\node[whitevertex] (W6) at (3,4) {};

\node[blackvertex] (B1) at (0,1) {};
\node[blackvertex] (B2) at (0,3) {};
\node[blackvertex] (B3) at (2,1) {};
\node[blackvertex] (B4) at (2,3) {};
\node[blackvertex] (B5) at (4,1) {};
\node[blackvertex] (B6) at (4,3) {};

\end{tikzpicture}
\\

\caption{}
\label{fig:mobt2}
\end{subfigure}
\begin{subfigure}{0.4\textwidth}
\centering
\begin{tikzpicture}[scale = 1,   
    whitevertex/.style={circle, draw, fill=black, minimum size=2mm, inner sep=0pt},
    blackvertex/.style={circle, draw, fill=white, minimum size=2mm, inner sep=0pt},
    dot/.style={dotted}, 
    line/.style={}
]

\draw[dashed] (0,0) -- (0,4) -- (4,4) -- (4,0) -- cycle;

\node[whitevertex] (W1) at (1,0) {};
\node[whitevertex] (W2) at (1,2) {};
\node[whitevertex] (W3) at (1,4) {};
\node[whitevertex] (W4) at (3,0) {};
\node[whitevertex] (W5) at (3,2) {};
\node[whitevertex] (W6) at (3,4) {};

\node[below] at (W1.south) {\scriptsize$A_2$};
\node[above] at (W2.north) {\scriptsize$A_3$};
\node[above] at (W3.north) {\scriptsize$A_1$};
\node[below] at (W4.south) {\scriptsize$A_1$};
\node[above] at (W5.north) {\scriptsize$A_4$};
\node[above] at (W6.north) {\scriptsize$A_2$};

\node[blackvertex] (B1) at (0,1) {};
\node[blackvertex] (B2) at (0,3) {};
\node[blackvertex] (B3) at (2,1) {};
\node[blackvertex] (B4) at (2,3) {};
\node[blackvertex] (B5) at (4,1) {};
\node[blackvertex] (B6) at (4,3) {};

\node[left] at (B1.west) {\scriptsize$\ell_3$};
\node[left] at (B2.west) {\scriptsize$\ell_1$};
\node[below] at (B3.south) {\scriptsize$\ell_4$};
\node[above] at (B4.north) {\scriptsize$\ell_2$};
\node[right] at (B5.east) {\scriptsize$\ell_3$};
\node[right] at (B6.east) {\scriptsize$\ell_1$};

\fill[gray!20, opacity=0.5] (W2.center) -- (B3.center) -- (W5.center) -- (B4.center) -- cycle;

\draw[line] (W1) -- (B1) node[midway,above] {\scriptsize$\quad b^{-1}$};
\draw[line] (B1) -- (W2) node[midway,below] {\scriptsize$1$};
\draw[line] (W2) -- (B2) node[midway,above] {\scriptsize$1$};
\draw[line] (B2) -- (W3) node[midway,below] {\scriptsize$1$};
\draw[line] (W3) -- (B4) node[midway,below] {\scriptsize$a$};
\draw[line] (B4) -- (W2) node[midway,above] {\scriptsize$a$};
\draw[line] (W2) -- (B3) node[midway,below] {\scriptsize$b$};
\draw[line] (B3) -- (W1) node[midway,above] {\scriptsize$1$};

\draw[line] (W4) -- (B3) node[midway,above] {\scriptsize$1$};
\draw[line] (B3) -- (W5) node[midway,below] {\scriptsize$a^{-1}$};
\draw[line] (W5) -- (B4) node[midway,above] {\scriptsize$\quad b^{-1}$};
\draw[line] (B4) -- (W6) node[midway,below] {\scriptsize$\,\, b^{-1}$};
\draw[line] (W6) -- (B6) node[midway,below] {\scriptsize$1$};
\draw[line] (B6) -- (W5) node[midway,above] {\scriptsize$1$};
\draw[line] (W5) -- (B5) node[midway,below] {\scriptsize$1$};
\draw[line] (B5) -- (W4) node[midway,above] {\scriptsize$a$};

\node[whitevertex] (W1) at (1,0) {};
\node[whitevertex] (W2) at (1,2) {};
\node[whitevertex] (W3) at (1,4) {};
\node[whitevertex] (W4) at (3,0) {};
\node[whitevertex] (W5) at (3,2) {};
\node[whitevertex] (W6) at (3,4) {};

\node[blackvertex] (B1) at (0,1) {};
\node[blackvertex] (B2) at (0,3) {};
\node[blackvertex] (B3) at (2,1) {};
\node[blackvertex] (B4) at (2,3) {};
\node[blackvertex] (B5) at (4,1) {};
\node[blackvertex] (B6) at (4,3) {};

\end{tikzpicture}
\caption{}
\label{fig:perm2}
\end{subfigure}

\caption{Tilings given by embeddings of $K_{4,4}$ into the torus. In both cases, opposite sides of the displayed fundamental domain are to be glued together  in such a way that vertices with the same labels are identified.}
\label{fig:torus-two-panel}

\end{figure}
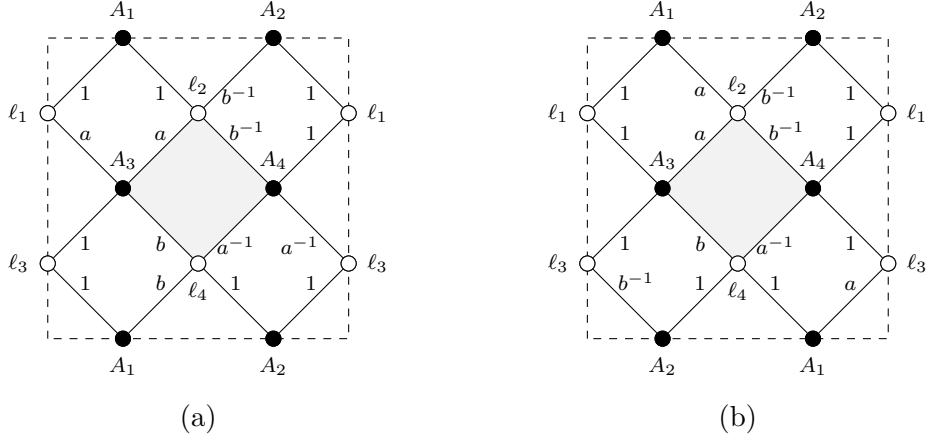

\begin{example}\label{truetoricthm}
The planar theorem corresponding to the tiling in Figure~\ref{fig:mobt2} is valid over any division ring. In fact, as explained in \cite[Example 5.18]{fomin2023incidences}, this theorem is rather trivial. Let us, however, explain the situation from the perspective of graph connections. Let $\D$ be a division ring, and assume we have a realization of the tiling from Figure~\ref{fig:mobt2} in $\P^2(\D)$ in which all tiles except the shaded one are coherent. We aim to show that the shaded tile is also coherent. To this end, consider the gauge-equivalence class of connections corresponding to the given realization. Within this class, we can always choose a connection as shown in the figure. Then the matrix $\Bell_j(\mathbf A_i)$ of pairings between vectors and covectors associated to the vertices is
\[
\begin{pmatrix}
1 & 1 & 1 & b \\
1 & b^{-1} & 1 & 1 \\
a & a & 1 & b \\
1 & b^{-1} & a^{-1} & a^{-1} \\
\end{pmatrix}
\]

Since the realization lies in $\P^2$, this matrix must be non-invertible. Using non-commutative Gaussian elimination, one finds that this occurs if and only if $a = 1$ or $b = 1$. In either case, $a$ and $b$ commute, so the holonomy of the connection around the shaded tile is trivial, which implies that the tile is coherent.

\end{example}

\begin{example}

Another embedding of $K_{4,4}$ into a torus is shown in Figure~\ref{fig:perm2}. The corresponding planar theorem is the permutation theorem (see \cite[Proof of Theorem 3.11]{fomin2023incidences}). Since this theorem can also be obtained from a triangulation of the sphere (see Example~\ref{permoct}), it holds over any division ring. This can also be seen from the perspective of graph connections.  Figure~\ref{fig:perm2} shows a connection that has trivial holonomy around all faces except for the shaded one. Any connection with this property is gauge-equivalent to the one shown. The associated matrix is
\[
\begin{pmatrix}
1 & a & a & 1 \\
1 & b^{-1} & b^{-1} & 1 \\
1 & a & 1 & b \\
1 & b^{-1} & 1 & a^{-1} 
\end{pmatrix}
\]
which is non-invertible if and only if either $a = b$ or $a = b^{-1}$. In both cases, $a$ and $b$ must commute, which implies that the associated planar theorem holds over any division ring.

\end{example}

Note that in all these examples, the associated theorem is either valid over all division rings, or holds if and only if the ring is a field.


\begin{problem}
Does there exist an incidence theorem associated with a tiling that holds over some, but not all, non-commutative division rings?
\end{problem}

It is known that there exist planar incidence theorems that hold, for example, in the ring of quaternions but fail in an arbitrary division ring. Such theorems clearly do not arise from any polygonal subdivision. Can a theorem of this type be realized via a tiling?

\section{Variation I: relative coherence}\label{sec:rel}

Given a division ring $\D$ and a normal subgroup $G \subset \D^\times$, one can also consider the following notion of \emph{coherence relative to $G$}:

\begin{definition}
Consider two points $A_1,A_2$ and two hyperplanes $\ell_1, \ell_2$ in the projective space $\P^n(\D)$. Assume that neither point lies on either hyperplane. Let $\mathbf A_1, \mathbf A_2 \in \D^{n+1}$ be vectors representing the points $A_1$ and $A_2$, and let $\Bell_1, \Bell_2 \in (\D^{n+1})^*$ be covectors representing the hyperplanes $\ell_1$ and $\ell_2$. The quadruple $A_1, A_2, \ell_1, \ell_2$ is said to be \emph{coherent relative to a normal subgroup $G \subset \D^\times$} if
\begin{equation*}
\Bell_1(\mathbf A_1)\,\Bell_1(\mathbf A_2)^{-1}
=
\Bell_2(\mathbf A_1)\,\Bell_2(\mathbf A_2)^{-1} \mod G.
\end{equation*}
Accordingly, given a realization of a tiling, a tile is said to be \emph{coherent relative to $G$} if the points $A_1,A_2$ assigned to its black vertices and the hyperplanes $\ell_1,\ell_2$ assigned to its white vertices form a quadruple coherent relative to $G$.
\end{definition}

Theorem~\ref{ST} holds for relative coherence without modification. 
Theorem~\ref{HGT} also remains valid, with the conclusion that the incidence theorem corresponding to $\mathcal T$ holds in $\P^n(\D)$ if and only if the quotient group $\D^\times / G$ is Abelian. 
The proofs of both theorems are obtained by replacing $\D^\times$-connections with connections valued in $\D^\times / G$.

\begin{example}
Let $\D = \mathbb H$ be the division ring of quaternions, and let $G = \mathbb R^\times$. Consider the projective line $\mathbb H \P^1 = S^4$. Since in dimension one points and hyperplanes coincide, it makes sense to speak about coherence of four points. By the above definition, four points $A,B,C,D \in \mathbb H \P^1$, with $A \neq B$, $B \neq C$, $C \neq D$, and $D \neq A$, are coherent relative to $G =\mathbb R^\times $ if their cross-ratio is real. This is equivalent to the points $A,B,C,D$ being cocyclic, see \cite[Proposition 12]{gwynne2012quaternionic}. As a result, Theorem \ref{ST} specializes to the following:

\begin{theorem}\label{CT}
Consider a quadrilateral tiling of the sphere and assign to each of its vertices a point in $S^4$ so that the points assigned to the endpoints of every edge are distinct. Assume that for all but one face of the tiling the points assigned to the vertices of the face are cocyclic. Then the points assigned to the vertices of the remaining face are also cocyclic.
\end{theorem}

For the tiling with the combinatorial structure of a cube, this is \cite[Theorem~3.2]{bobenko2008}, which lies at the heart of the construction of \emph{circular nets}. The planar case of Theorem~\ref{CT} is \cite[Corollary~11.6]{fomin2023incidences}. This planar case corresponds to $\D=\C$ and $G=\mathbb R^\times$. Since in this case $\D$ is commutative, the planar theorem also holds for tilings of higher-genus surfaces. 


\end{example}

\section{Variation II: arbitrary rings}
The results of this paper also extend, with suitable modifications, to rings that do not necessarily admit division. See \cite{hand2026pentagram} for a general overview of projective geometry over arbitrary rings. We note that in this general setting there are no universally accepted definitions of standard objects in projective geometry. The definition given below is adapted to the setting of the present paper.

\begin{definition}
Let $\R$ be a ring and $M$ an $R$-module. A \emph{point} in the associated projective space $\P(M)$ is a free rank-one submodule $A \subset M$ that is a direct summand of $M$, meaning that there exists another submodule $\ell \subset M$ such that $M = A \oplus \ell$. A \emph{hyperplane} in $\P(M)$ is a submodule $\ell \subset M$ such that $M = A \oplus \ell$ for some point $A$. A point $A \subset M$ and a hyperplane $\ell \subset M$ are said to be \emph{non-neighboring} if $A \oplus \ell = M$, cf.~\cite{veldkamp1967}.

If $M$ is a free (left) $R$-module of rank $n+1$, then $\P(M)$ is called the (left) \emph{$n$-dimensional projective space over $\R$} and is denoted by $\P^n(\R)$.
\end{definition}

\begin{example}
Let $\R = \mathrm{M}_k(\F)$ be the ring of $k \times k$ matrices over a field $\F$. Then any left $R$-module is of the form $V^k \simeq V \otimes \F^k$ for some vector space $V$, where $\mathrm{M}_k(\F)$ acts on $V \otimes \F^k$ by acting on the second tensor factor. Assuming that $V$ has finite dimension $n$ and choosing a basis, one can identify $V^k$ with the space of $k \times n$ matrices, on which $\mathrm{M}_k(\F)$ acts by left multiplication.

The correspondence $W \mapsto W^k$, where $W \subset V$, identifies the lattice of subspaces of $V$ with the lattice of submodules of $V^k$. A subspace $W \subset V$ gives rise to a free rank-one submodule if and only if $\dim W = k$. Thus, for $M = V^k$, points in $\P(M)$ correspond to subspaces of $V$ of dimension $k$, and hyperplanes in $\P(M)$ correspond to subspaces of $V$ of codimension $k$. A point and a hyperplane are non-neighboring if and only if the corresponding subspaces are transversal; see~\cite[Section~3.2]{hand2026pentagram}.

In particular, $\P^n(\mathrm{M}_k(\F))$ can be identified with the Grassmannian $\mathrm{Gr}(k, k(n+1))$, while hyperplanes in $\P^n(\mathrm{M}_k(\F))$ correspond to points of the dual Grassmannian $\mathrm{Gr}(kn, k(n+1))$.
\end{example}

\begin{example}
Let $\R = \mathbb{R}[\varepsilon] / (\varepsilon^2)$ be the ring of \emph{dual numbers}. Then $\P^2(\R)$ can be identified with the set of lines in Euclidean space $\mathbb{R}^3$. Given a line through a point $p \in \mathbb{R}^3$ with direction vector $v \in \mathbb{R}^3$, the corresponding point in $\P^2(\R)$ is the submodule spanned by $v + \varepsilon v \times p$, where $v \times p$ denotes the cross product; see~\cite[Section~3.5]{hand2026pentagram}.

Hyperplanes in $\P^2(\R)$ can also be identified with lines in $\mathbb{R}^3$. (Although $\P^2(\R)$ is a projective plane, we avoid calling hyperplanes “lines” so as not to confuse them with Euclidean lines.) Given a line through a point $p$ with direction vector $v$, the corresponding hyperplane in $\P^2(\R)$ is the submodule $(v + \varepsilon v \times p)^\bot$, where the orthogonal complement is taken with respect to the natural extension of the Euclidean inner product on $\mathbb{R}^3$ to $\R^3$.

A point and a hyperplane in $\P^2(\R)$ are incident if and only if the corresponding lines in $\mathbb{R}^3$ meet at a right angle. A point and a hyperplane in $\P^2(\R)$ are neighboring if and only if the direction vectors of the corresponding lines in $\mathbb{R}^3$ are orthogonal.
\end{example}

%
%

\begin{definition}[cf. Definition \ref{def:realization}] Let $\R$ be an arbitrary ring, and $M$ be an $R$-module. 
A \emph{realization} of a tiling $\mathcal T$ in $\P(M)$ is an assignment of a point in $\P(M)$ to each black vertex of $\mathcal T$ and a hyperplane in $\P(M)$ to each white vertex, such that for every edge of $\mathcal T$ the point and the hyperplane assigned to its endpoints are non-neighboring.
\end{definition}

\begin{definition}[cf. Definiton \ref{def:coh}]
Let $\R$ be a ring, and $M$ be a left $R$-module. Consider two points $A_1, A_2$ and two hyperplanes $\ell_1, \ell_2$ in the projective space $\P(M)$. Assume that all point-hyperplane pairs are non-neighboring. 
Then the quadruple $A_1, A_2, \ell_1, \ell_2$ is said to be \emph{coherent} if
$$
(A_1 + A_2) \cap \ell_1 = (A_1 + A_2) \cap \ell_2.
$$

Accordingly, given a realization of a tiling, a tile is said to be \emph{coherent} if the points $A_1, A_2$ assigned to its black vertices and the hyperplanes $\ell_1, \ell_2$ assigned to its white vertices form a coherent quadruple.
\end{definition}

\begin{example}
Consider an $\F$-vector space $V$. Let $A_1, A_2 \subset V$ be subspaces of dimension $k$, and $\ell_1, \ell_2 \subset V$ be subspaces of codimension $k$ transversal to $A_1$ and $A_2$. Then $A_1, A_2$, viewed as points in the projectivization $\P(V^k)$ of the left $\mathrm{M}_k(\F)$-module $V^k$, and $\ell_1, \ell_2$, viewed as hyperplanes in $\P(V^k)$, are coherent if and only if
$$
(A_1 + A_2) \cap \ell_1 = (A_1 + A_2) \cap \ell_2.
$$

\end{example}

We now reformulate the coherence condition algebraically. Given a point $A \in \P(M)$, its \emph{lift} is any $\mathbf A \in M$ that generates $A$ as a submodule. Likewise, given a hyperplane $\ell$, its lift is any $\Bell \in M^*$ generating the annihilator of $\ell$; such a generator exists because the definition of a hyperplane implies that annihilator of $\ell$ is a point in $\P(M^*)$. A point $A$ and a hyperplane $\ell$ in $\P(M)$ are non-neighboring if and only if, for their lifts $\mathbf A \in M$ and $\Bell \in M^*$, the element $\Bell(\mathbf A)$ is a unit in $\R$.

\begin{proposition}[cf. Proposition \ref{propCoh}]\label{propCoh2}
Let $\R$ be a ring, and let $M$ be a left $\R$-module. Consider two points $A_1, A_2$ and two hyperplanes $\ell_1, \ell_2$ in the projective space $\P(M)$, and assume that all point--hyperplane pairs are non-neighboring. Let $\mathbf A_1, \mathbf A_2 \in M$ be lifts of $A_1$ and $A_2$, and let $\Bell_1, \Bell_2 \in M^*$ be lifts of $\ell_1$ and $\ell_2$. Then $A_1, A_2, \ell_1, \ell_2$ are coherent if and only if their lifts satisfy \eqref{eq2st}.
\end{proposition}

\begin{proof}
Assume that $A_1, A_2, \ell_1, \ell_2$ are coherent, and set
\begin{equation}\label{eq:boldA}
\mathbf A :=
\Bell_1(\mathbf A_1)^{-1}\mathbf A_1 - \Bell_1(\mathbf A_2)^{-1}\mathbf A_2.
\end{equation}
Then $\Bell_1(\mathbf A)=0$, so $\mathbf A \in \ell_1 \cap (A_1 + A_2)$. By coherence, $\mathbf A \in \ell_2 \cap (A_1 + A_2)$, hence $\Bell_2(\mathbf A)=0$. On the other hand,
\[
\Bell_2(\mathbf A)
=
\Bell_1(\mathbf A_1)^{-1}\Bell_2(\mathbf A_1)
-
\Bell_1(\mathbf A_2)^{-1}\Bell_2(\mathbf A_2),
\]
so $\Bell_2(\mathbf A)=0$ yields \eqref{eq2st}. Conversely, assume that \eqref{eq2st} holds, and define $\Bell \in M^*$ by
\begin{equation}\label{eq:boldL}
\Bell(\mathbf A) := \Bell_1(\mathbf A)\Bell_1(\mathbf A_2)^{-1} - \Bell_2(\mathbf A)\Bell_2(\mathbf A_2)^{-1}.
\end{equation}
Then $\Bell(\mathbf A_2)=0$ by definition, and $\Bell(\mathbf A_1)=0$ by \eqref{eq2st}. Hence $\Bell$ vanishes on $A_1 + A_2$, which implies
\[
\Bell_1(\mathbf A)\Bell_1(\mathbf A_2)^{-1}
=
\Bell_2(\mathbf A)\Bell_2(\mathbf A_2)^{-1}
\]
for all $\mathbf A \in A_1 + A_2$. Therefore, $\Bell_1(\mathbf A)=0$ if and only if $\Bell_2(\mathbf A)=0$ for all $\mathbf A \in A_1 + A_2$, which is equivalent to the definition of coherence.
\end{proof}
\begin{remark}\label{rem:dualcoh}
Dually, one can define coherence by the condition
\begin{equation}\label{eq:dc}
A_1 + (\ell_1 \cap \ell_2) = A_2 + (\ell_1 \cap \ell_2).
\end{equation}
Its equivalence to \eqref{eq2st} is shown as follows.  Assume \eqref{eq2st} holds and let $\mathbf A \in M$ be given by~\eqref{eq:boldA}. Then $\mathbf A \in \ell_1 \cap \ell_2$, which implies $\mathbf A_1 \in A_2 + (\ell_1 \cap \ell_2)$ and $\mathbf A_2 \in A_1 + (\ell_1 \cap \ell_2)$, hence~\eqref{eq:dc} holds. Conversely, assume \eqref{eq:dc} and let $\Bell \in M^*$ be defined by \eqref{eq:boldL}. Then $\Bell$ vanishes on $\ell_1 \cap \ell_2$ and on $A_2$. Since \eqref{eq:dc} implies $A_1 \subset A_2 + (\ell_1 \cap \ell_2)$, it follows that $\Bell$ also vanishes on $A_1$, yielding \eqref{eq2st}.
\end{remark}

\begin{example}
Consider four lines $A_1, A_2, \ell_1, \ell_2$ in Euclidean space $\mathbb{R}^3$. Assume that the direction vector of $\ell_i$ is not perpendicular to that of $A_j$ for all $i,j=1,2$, and denote by $S(\cdot,\cdot)$ the common perpendicular of two skew lines. Using the characterizations of coherence obtained above, one can show that $A_1, A_2$ and $\ell_1, \ell_2$, viewed respectively as points and hyperplanes in the projective plane over the dual numbers, are coherent if and only if one of the following holds:
\begin{itemize}
\item $A_1 \parallel A_2$ and $\ell_1 \parallel \ell_2$;
\item $A_1$ and $A_2$ are skew, and
\[
S(S(A_1, A_2), \ell_1) = S(S(A_1, A_2), \ell_2),
\]
(note that under our assumptions, if $A_1$ and $A_2$ are skew, then $S(A_1,A_2)$ is automatically skew to $\ell_1$ and $\ell_2$, so the lines $S(S(A_1, A_2), \ell_i)$ are well defined);
\item $\ell_1$ and $\ell_2$ are skew, and
\[
S(S(\ell_1, \ell_2), A_1) = S(S(\ell_1, \ell_2), A_2),
\]
(under our assumptions, if $\ell_1$ and $\ell_2$ are skew, then $S(\ell_1,\ell_2)$ is automatically skew to $A_1$ and $A_2$, so the lines $S(S(\ell_1, \ell_2), A_i)$ are well defined).
\end{itemize}

When both pairs $A_1, A_2$ and $\ell_1, \ell_2$ are skew, coherence simply means that their common perpendiculars $S(A_1, A_2)$ and $S(\ell_1, \ell_2)$ meet at a right angle. This is a direct analog of the coherence condition $\ell_1 \cap \ell_2 \in A_1A_2$ in projective planes over division rings.\end{example}

Since coherence over arbitrary rings corresponds to the same algebraic condition \eqref{eq2st} as in the case of division rings, our theorems extend to this more general setting. In particular, Theorem~\ref{ST} holds for an arbitrary ring $\R$ without modification. Furthermore, if $\R$ is commutative, the assumption that the tiling lies on the sphere can be dropped; cf.~\cite{kuhne2025absolute}.

%
%
%
%
\begin{example}

The \emph{correspondence principle} of \cite{tabachnikov2018skewers} asserts that any planar incidence theorem gives rise to an incidence theorem about lines in $\mathbb R^3$ as follows: replace all points and lines by lines in $\mathbb R^3$, and replace incidence between a point and a line by the condition that the corresponding lines meet at a right angle. Applying Theorem~\ref{ST} to the ring of dual numbers yields the correspondence principle for incidence theorems arising from tilings (the genus of the surface is irrelevant since the ring of dual numbers is commutative).


\end{example}

\begin{example}
In any theorem associated with a spherical tiling, one can replace points and hyperplanes with, respectively, dimension $k$ and codimension $k$ subspaces. This corresponds to taking the ground ring to be $\mathrm M_k(\F)$. In this case the genus is important, since the ring of matrices is noncommutative for $k > 1$.

\end{example}

 We also have the following version of Theorem \ref{HGT}:

\begin{theorem}\label{HGTR}
Let $\mathcal T$ be a simple tiling of a surface of positive genus, and let $\R$ be a ring. Suppose that
$
n \ge |B(\mathcal T)| - 1,
$
where $B(\mathcal T)$ denotes the set of black vertices of $\mathcal T$. Then the incidence theorem corresponding to $\mathcal T$ holds in $\P^n(\R)$ if and only if the group of units $\R^\times$ is commutative.
\end{theorem}

Note that in this setting we need a stronger assumption on the dimension, because Corollary~\ref{corinf} does not hold for arbitrary rings. To prove Theorem~\ref{HGTR}, one proceeds as in the proof of Theorem~\ref{HGT} but, instead of assigning to black vertices vectors from the set coming from Corollary~\ref{corinf}, one simply assigns a distinct basis vector to each black vertex. This is possible because of the assumption $n \ge |B(\mathcal T)| - 1$.

\begin{example}
Applying this theorem to the tiling from Figure \ref{fig:pappust}, we see that the Pappus theorem fails over any ring with non-Abelian group of units.  
\end{example}

\begin{remark} Theorem~\ref{HGTR} remains true assuming only that
$
n \ge \max_{v \in W(\mathcal T)} \bigl( \deg v \bigr) - 1,
$
provided that Corollary~\ref{corinf} holds over the ring $\R$. For example, this is the case for rings that are algebras over an infinite field $\F$. In that case, the set $S$ can be taken to consist of vectors with coordinates $(1, \dots, \lambda_i^{k-1})$, where $\lambda_i \in \F$ are distinct.

\end{remark}

Theorems~\ref{ST} and~\ref{HGTR} also admit relative versions over arbitrary rings, where tiles are assumed to be coherent only modulo a normal subgroup $G \subset \R^\times$, as in Section~\ref{sec:rel}.
\begin{example}
Let $\R = \mathrm{M}_2(\F)$. Then points in $\P^1(\R)$ correspond to lines in $\P^3(\F)$. Consider the subgroup $G \subset \R^\times$ consisting of scalar matrices. Then coherence relative to $G$ for four lines $\ell_1, \dots, \ell_4 \subset \P^3(\F)$ is equivalent to the existence of infinitely many lines simultaneously meeting $\ell_1, \dots, \ell_4$, or, equivalently, to the condition that $\ell_1, \dots, \ell_4$ lie in a single ruling of a quadric surface. This yields the following version of Theorem~\ref{CT}:

\begin{theorem}\label{CT2}
Consider a quadrilateral tiling of the sphere. Assign to each vertex a line in $\P^3$ so that the lines assigned to the endpoints of every edge are skew. Assume that for all but one face of the tiling, the four lines assigned to its vertices lie in a single ruling of a quadric surface. Then the lines assigned to the vertices of the remaining face also lie in a single ruling of a quadric surface.
\end{theorem}\end{example}
\begin{example}[M. Skopenkov, private communication]
Let $\R$ be the ring of dual numbers, and $G = \mathbb{R}$. Identify the affine patch in $\P^1(\R)$ with $\mathbb{R}^2$. Then four points in that affine patch, no two of which lie on a vertical line, are coherent relative to $G$ if and only if they lie on the same Galilean circle, i.e., a parabola of the form $y = ax^2 + bx + c$; indeed, the cross-ratio of four dual numbers with pairwise distinct real parts is real if and only if the corresponding points lie on the same Galilean circle, see \cite[Appendix B]{yaglom2012simple}.

Likewise, if we take $\R$ to be the ring of double numbers, $\R = \mathbb{R}[x]/(x^2 - 1)$, then coherence in $\P^1(\R)$ relative to $\mathbb{R}$ means that the points lie on the same Minkowskian circle, i.e., a circle with respect to the metric of indefinite signature in $\mathbb{R}^2$ (here we assume that no two of the four points lie on the same isotropic line, i.e., a line of slope $\pm 1$).

This gives Galilean and Minkowski versions of the planar case of Theorem~\ref{CT}. That theorem holds for tilings of surfaces of any genus, because the corresponding rings are commutative.
\end{example}

%
%
%
%

\bibliographystyle{plain}
\bibliography{inc.bib}

\end{document}